\documentclass[a4paper, 12pt,reqno]{amsart}

\usepackage{graphicx}
\usepackage{amsmath, amsfonts, amssymb, amsthm, xspace}
\usepackage{xcolor}
\usepackage[margin=0.75in]{geometry}

\usepackage{booktabs} %\toprule, \midrule \addlinespace
\usepackage{enumitem} % adjust spacings
\makeatletter
\def\namedlabel#1#2{\begingroup
    #2%
    \def\@currentlabel{#2}%
    \phantomsection\label{#1}\endgroup
}
\makeatother

\usepackage[utf8]{inputenc}
\renewcommand{\pmod}[1]{\left( \mathrm{ mod\;}#1\right)}

\newcommand{\Area}{\operatorname{Area}}

\usepackage{mathtools}

\usepackage{mathrsfs}

\usepackage{oldgerm}

\newcommand{\SF}{\mathfrak{SF}}

\renewcommand{\SF}{\mathfrak{S}}
\renewcommand{\SF}{\mathscr S}
\newcommand{\cF}{\mathcal{F}}
\newcommand{\cH}{\mathcal H}

\newcommand{\cT}{\mathcal T}

\newcommand{\cS}{\mathcal S}

\newcommand{\cQ}{\mathcal Q}
\newcommand{\cE}{\mathcal E}

\newcommand{\cG}{\mathcal G}

\newcommand{\ZZ}{\mathbb{Z}}
\newcommand{\RR}{\mathbb{R}}

\newcommand{\NN}{\mathbb{N}}
\newcommand{\QQ}{\mathbb{Q}}
\newcommand*\dd{\mathop{}\!\textnormal{\slshape d}}

\newcommand{\SL}{\operatorname{SL}(2,\NN)}

\newcommand{\va}{\bar{a}}
\newcommand{\vq}{\bar{q}}
\newcommand{\vva}{\bar{\bar{a}}}
\newcommand{\vvq}{\bar{\bar{q}}}
\newcommand{\vvva}{\bar{\bar{\bar{a}}}}

\newcommand{\vd}{\bar{d}}

\usepackage{bm}

\newtheorem{theorem}{Theorem}%[section]
\newtheorem{lemma}[theorem]{Lemma}
\newtheorem{corollary}[theorem]{Corollary}

\newtheorem{proposition}[theorem]{Proposition}

\AddToHook{cmd/appendix/before}{%
  \setcounter{equation}{0}%
  \setcounter{theorem}{0}%
  \renewcommand{\thetheorem}{\thesection.\arabic{theorem}}%
}

\theoremstyle{remark}

\newtheorem{remark}{Remark}

\usepackage{subfigure}  % package for subfigure formatting
\usepackage{float}
\usepackage[]{caption}
\newcommand{\sdfrac}[2]{\mbox{\small$\displaystyle\frac{#1}{#2}$}}

\definecolor{orange}{rgb}{1,0.5,0}
\definecolor{Red}{rgb}{.795,0.015,0.017}
\definecolor{Ggreen}{rgb}{0.,0.675,0.0128}
\definecolor{Bblue}{rgb}{0.16,.32,0.91}

\usepackage[hyphens]{url}
\usepackage{multicol}

 \usepackage[backref=page]{hyperref}
\hypersetup{
    colorlinks = true,
linkcolor={Red},
urlcolor={blue},
citecolor={Ggreen},    
urlcolor = {blue},
citebordercolor = {0.33 .58 0.33},
 linkbordercolor = {0.99 .28 0.23},
 breaklinks=true
}
\renewcommand*{\backref}[1]{}
\renewcommand*{\backrefalt}[4]{%
  \ifcase #1 %
No citations.% use \relax if you do not want the "No citations" message
  \or
(page #2).%
  \else
(pages #2).%
  \fi%
}
\usepackage{cite}% condensed citations
\begin{document}

\title[On the distribution of $\SL$-saturated Farey fractions]{On the distribution of $\SL$-saturated
Farey fractions}
\author{Jack Anderson, Florin P. Boca, Cristian Cobeli, Alexandru Zaharescu}

\date{\today}

%%%%%%%%%%%%%%%%%%%%%%%

% \author[Jack Anderson]{Jack Anderson}
\address[Jack Anderson]{Department of Mathematics, University of Illinois at Urbana-Champaign, Urbana, IL 61801, USA.}
% \address[Jack Anderson]{Department of Mathematics, University of Illinois, 1409 West Green 
% Street, Urbana, IL 61801, USA.}
\email{jacka4@illinois.edu}

% \author[Florin P. Boca]{Florin P. Boca}
\address[Florin P. Boca]{Department of Mathematics, University of Illinois at Urbana-Champaign, Urbana, IL 61801, USA.}
% \address[Florin P. Boca]{Department of Mathematics, University of Illinois, 1409 West Green 
% Street, Urbana, IL 61801, USA.}
\email{fboca@illinois.edu}

% \author[Cristian Cobeli]{Cristian Cobeli}
\address[Cristian Cobeli]{"Simion Stoilow" Institute of Mathematics of the Romanian Academy,~21 Calea Grivitei Street, P. O. Box 1-764, Bucharest 014700, Romania}
\email{cristian.cobeli@imar.ro}

% \author{Alexandru Zaharescu}
\address[Alexandru Zaharescu]{Department of Mathematics,University of Illinois at Urbana-Champaign, Urbana, IL 61801, USA,
% \address[Alexandru Zaharescu]{Department of Mathematics, University of Illinois, 1409 West Green 
% Street, Urbana, IL 61801, USA,
% .}
%
% \address[Alexandru Zaharescu]{
and 
"Simion Stoilow" Institute of Mathematics of the Romanian Academy,~21 
Calea Grivitei 
Street, P. O. Box 1-764, Bucharest 014700, Romania}
\email{zaharesc@illinois.edu}


% {Primary 11B37; Secondary 11B39, 11B50.}
% 11B99 11B85
% 11Bxx Sequences and sets
% 11B99 None of the above, but in this section
% 05A05 Permutations, words, matrices
% 68R15 Combinatorics on words
% 11M41 Other Dirichlet series and zeta functions 


\begin{abstract}
We consider the set $\SF_Q$ of Farey fractions $d/b$ of order $Q$ with the property that there exists 
a matrix $\left( \begin{smallmatrix} a &  b \\ c & d \end{smallmatrix} \right) \in \operatorname{SL}(2,{\mathbb Z})$  of trace at most $Q$, 
with positive entries and
\mbox{$a\ge \max\{ b,c\}$}.
For every $Q\ge 3$, the set $\SF_Q \cup \{ 0\}$ is shown to define a unimodular partition of the interval~$[0,1]$.
We also prove that the elements of $\SF_Q$ are asymptotically distributed with respect to the probability measure with density
$ (1/(1+x) -1/(2+x) )/\log (4/3) $ and that the sequence
of sets~$(\SF_Q)_Q$ has a limiting gap distribution as $Q\rightarrow \infty$.
\end{abstract}
\maketitle

%%%%%%%%%%%%%%%%%%%%%%%%%%%%%%%%%%%%%%%%%%%%%%
\section{Introduction}
The monoid
\begin{equation*}
\cS :=\left\{ M=\left( \begin{matrix} a & b \\ c & d \end{matrix}\right)  \in \operatorname{SL}(2,\ZZ), \
a\ge b \ge d\ge 1,\ a\ge c \ge d\right\}
\end{equation*}
appears naturally in Diophantine approximation.
An immediate application of the Euclidean algorithm shows that these matrices can be uniquely expressed as products
\begin{equation*}
M=
\left( \begin{matrix} a_1 & 1 \\ 1 & 0 \end{matrix}\right) \cdots \left( 
\begin{matrix} a_{2r} &  1 \\1 & 0 \end{matrix} \right)
=\left( \begin{matrix} q_{2r+1} & q_{2r} \\ p_{2r+1} & p_{2r} \end{matrix} \right) ,
\quad r, a_i \in \NN,
\end{equation*}
while $c/a=p_{2r+1}/q_{2r+1}$ and $d/b=p_{2r}/q_{2r}$ represent consecutive convergents of the numbers of the  form
$[0;a_1,a_2,\ldots,a_{2r},\ast] \in [0,1] \setminus \QQ$.

The sets
\begin{equation*}
 \cS_Q :=\{ M\in\cS : \operatorname{Tr} M \le Q\} 
\end{equation*}
play an important role, for instance, in the study of the distribution of reduced quadratic irrationals
(which coincide with the periodic points of the Gauss map for regular continued fractions)
\cite{KOPS2001,Bo2007,Ust2013}, or of periodic points of the Farey map (see the appendix to 
\cite{Hee2019}). 
It is known  (see \cite{Ust2013} for the final result) that
\begin{equation*}
\# \cS_Q =\frac{\log 2}{2\zeta(2)} \, Q^2 +O (Q^{3/2}\log^4 Q),
\end{equation*}
and more generally, for every $\beta \in [0,1]$,
\begin{equation}\label{eq1}
\#  \left\{ \left( \begin{matrix} a & b \\ c & d \end{matrix}\right) \in \cS_Q : \frac{d}{b} \le \beta \right\} 
=\frac{\log(1+\beta)}{2\zeta(2)} \, Q^2 +O (Q^{3/2}\log^4 Q) .
\end{equation}
In particular, one recovers a version of Pollicott's classical distribution result for reduced quadratic
irrationals \cite{Pol1986}
\begin{equation*}
\lim\limits_{Q\rightarrow \infty} \frac{\# \left\{ \left( \begin{smallmatrix} a & b \\ c & d \end{smallmatrix} \right)  \in \cS_Q:
d/b \le \beta \right\}}{\# \cS_Q} =\int_0^\beta \frac{dx}{(1+x)\log 2} .
\end{equation*}

Considering the natural map 
\begin{equation*}
\Psi :\cS \rightarrow \QQ \cap (0,1],\quad \Psi \bigg( \left( \begin{matrix} a& b \\ c & d \end{matrix} \right) \bigg) :=\frac{d}{b} ,
\end{equation*}
we study the asymptotic distribution of \emph{$\SL$-saturated Farey fractions of order $Q$},
that we define as the elements of the subset
\begin{equation*}
\SF_Q :=\Psi (\cS_Q) 
\end{equation*}
of the customary Farey set
\begin{equation*}
\cF_Q :=\left\{ \frac{d}{b} : d,b\in \NN, \  0\le d \le  b \le  Q,\   (d,b)=1 \right\} .
\end{equation*}
The sets $\SF_Q$ exhibit interesting arithmetic properties, resembling those of $\cF_Q$. Firstly, their union gives 
precisely the set ${\mathbb Q} \cap (0,1]$. Secondly, as shown in the Appendix,
the elements of each set $\SF_Q^* :=\{ 0\} \cup \SF_Q$ define a {\it unimodular partition} of the interval $[0,1]$.
This means that if $0/1 < a_1/q_1= 1/Q < \cdots < a_{N(Q)}/b_{N(Q)}$ denote the elements 
of $\{ 0\} \cup \SF_Q$, then $a_{i+1} q_i -a_i  q_{i+1} =1$ for every $i=1,\ldots, N(Q) -1$.
As an interesting geometric application, when viewed as elements on the
boundary $\RR \cup \{\infty\}$ of the upper half-plane $\mathbb{H}$, consecutive elements in $\SF_Q^*$ are
connected by arcs in the Farey tessellation.
Thirdly, it is observed in the beginning of the proof of Theorem~\ref{theorem1} that
\begin{equation*}
\SF_Q =\{ d/b \in \cF_Q : b+d+\bar{d} \le Q \},
\end{equation*}
where $\bar{d}$ denotes the multiplicative inverse of $d\pmod{b}$ in the interval $[1,b)$.

However, the elements of $\SF_Q$ are not uniformly distributed as $Q\rightarrow\infty$.
Their asymptotic distribution 
is investigated in Section~\ref{Section2}, where we prove

\begin{theorem}\label{theorem1}
For every $\beta \in [0,1]$,
\begin{equation*}
\# ( \SF_Q \cap [0,\beta]) =\frac{Q^2}{2\zeta (2)}
\log \bigg( \frac{2(1+\beta)}{2+\beta}\bigg) +O_\varepsilon (Q^{7/4+\varepsilon}) ,
\quad Q\rightarrow \infty.
\end{equation*}
\end{theorem}

\begin{corollary}\label{cor2}
$\displaystyle N(Q):=\# \SF_Q =\frac{Q^2}{2\zeta (2)} \log (4/3) +O_\varepsilon (Q^{7/4+\varepsilon}).$
\end{corollary}

%In particular, the asymptotic 
%density of $\SF_Q$ in $\cF_Q$ is given by
%\begin{equation*}
%\lim\limits_{Q\rightarrow\infty} \frac{\# \SF_Q}{\# \cF_Q} =\log (4/3)  \approx 0.28768.
%\end{equation*}

\begin{corollary}\label{cor3}
For every $\beta \in [0,1]$,
\begin{equation*}
\lim\limits_{Q\rightarrow \infty} \frac{\# (\SF_Q \cap[0,\beta])}{\# \SF_Q} =\int_0^\beta 
\bigg( \frac{1}{1+x} -\frac{1}{2+x}\bigg) \frac{dx}{\log (4/3)}  .
\end{equation*}
\end{corollary}

%The convergence in Corollary~\ref{cor3} is numerically illustrated in Figure~\ref{approx}. 

%\begin{figure}[ht]
 %\centering    
% \includegraphics[angle=0,width=0.45\textwidth]{G3.pdf}
%\quad
% \includegraphics[angle=0,width=0.45\textwidth]{G4.pdf}
% \vspace*{-10pt}
%\caption{The convergence in Corollary~\ref{cor3} illustrated for $Q=50$ and $Q=300$}
% \label{approx}
% \end{figure}

%For instance, if $Q=400$, 
%the ratio is $\#\SF_Q/\#\cF_Q\approx 0.28704$,
%because
%$\#\SF_Q=13\,973$ and
%\mbox{$\#\cF_Q=48\,679$},
%while
%$\log(4/3) \approx 0.28768$.
% # Q = 400
% # card(Farey(Q)) =  48679
% # card(Cadmissible(Q)) =  13973
% # 0.28704369440621214 == 0.287043694406212
% # log(4/3) =-= 0.287682072451781

It is well known  that the classical Farey sequence $(\cF_Q)_Q$  is uniformly distributed on $[0,1]$ (cf., e.g., \cite[Theorem~1]{Mik1949}). 
On the other hand, the repulsion between rational numbers, illustrated by the inequality 
$\lvert \gamma -\gamma^\prime \rvert \ge 1/Q^2$ if $\gamma,\gamma^\prime \in \cF_Q$, $\gamma \neq \gamma^\prime$,
ensures that the spacing statistics of $(\cF_Q)_Q$ are not Poissonian.
The Farey statistics have been thoroughly analyzed  (cf., e.g., \cite{Hall1970,HT1984,KZ1997,ABCZ2001,BCZ2001,BZ2005,AC2014}), 
even for more geometrically involved analogues, such as in \cite{Ma2013,He2021,Zhang2021,Lu2022}.
Although the $\SL$-saturated Farey fractions are not uniformly distributed, 
one can investigate the corresponding spacing statistics but problems become more challenging. 
As a first step in this direction, we establish in Section~\ref{Section3} the existence of the limiting gap distribution of $(\SF_Q)_Q$, proving
the following result.
\begin{theorem}\label{thm4}
The gap distribution $\cG (\lambda):=\lim_{Q\rightarrow \infty} \cG_Q(\lambda)$,
where
\begin{equation*}
\cG_Q(\lambda) :=\frac{\#\{ \gamma_1 <\gamma_2\ \text{\upshape consecutive in $\SF_Q$}: \gamma_2 -\gamma_1 \le \lambda/N(Q)\}}{N(Q)},\qquad \lambda \ge 0,
\end{equation*}
exists and can be expressed explicitly in terms of iterates of the map $T$ which acts on the triangle 
$\Omega :=\{ (x,y): 0<x,y \le 1, x+y >1\}$ as
\begin{equation}\label{eq2}
T(x,y)=\big( y,\kappa (x,y)y-x\big) , \quad \text{with} \quad \kappa (x,y):=\bigg\lfloor \frac{1+x}{y}\bigg\rfloor.
\end{equation}
\end{theorem}

The shape of the graphs of  $\cG$ and its derivative (gap density) are illustrated in Figure~\ref{FigureGG}.

\begin{figure}[ht]
 \centering    
 \includegraphics[angle=0,width=0.45\textwidth]{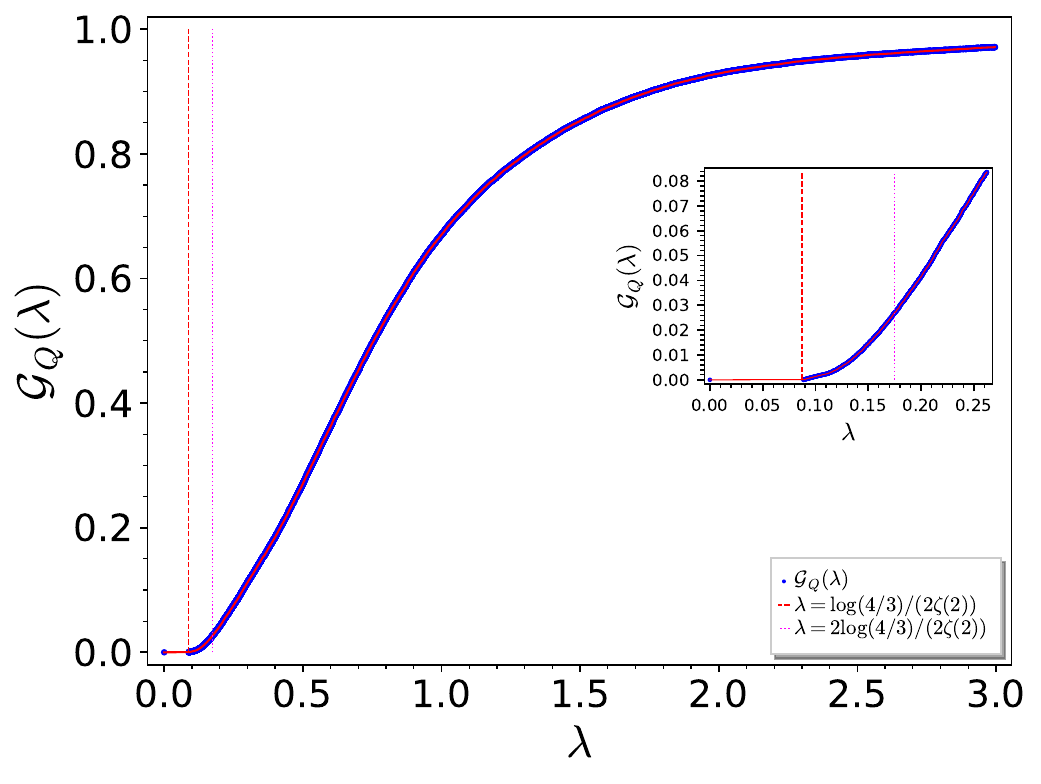}
\quad
 \includegraphics[angle=0,width=0.45\textwidth]{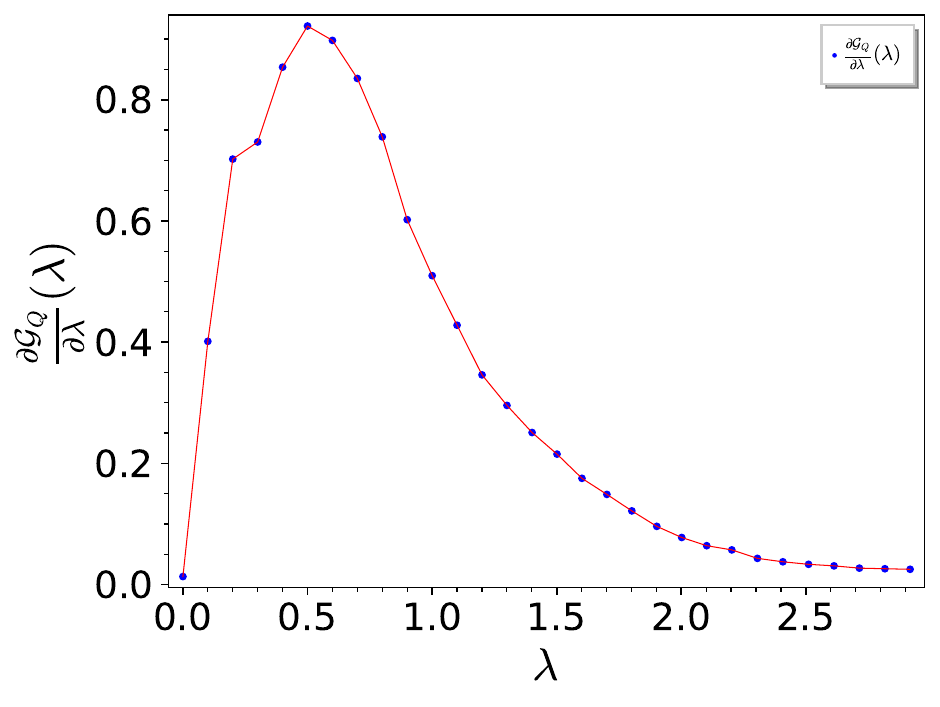}
% \
%  \includegraphics[angle=0,width=0.49\textwidth]{Gte082.pdf}
 \vspace*{-10pt}
\caption{The graphs of $\cG_{500} (\lambda)$ and $\frac{\cG_{500} (\lambda+0.1)-\cG_{500} (\lambda)}{0.1}$}
 \label{FigureGG}
 \end{figure}

Furthermore, if we set $A:=\log(4/3)/(2\zeta(2))$, then $\cG(\lambda)=0$ for every $\lambda \in [0,A]$.
A similar property holds for the gap distribution density of $(\cF_Q)_Q$, which vanishes on the interval $[0,1/(2\zeta(2))]$.
When $\lambda \in (A,2A]$, a closed-form expression of $\cG (\lambda)$ is obtained at the end of Section~\ref{Section3}, via~\eqref{eq9}.

\section{\texorpdfstring{The asymptotic distribution of elements in the sets $\cS_Q$ and $\SF_Q$}
{The asymptotic distribution of elements in the sets SQ and CQ} }\label{Section2}
To illustrate our general strategy, we start by reproving the estimate~\eqref{eq1} with the weaker error term $O_\varepsilon (Q^{7/4+\varepsilon})$.
Given a matrix $\left( \begin{smallmatrix} a & b \\ c & d \end{smallmatrix} \right) \in \cS_Q$ and letting $\va$ denote 
the multiplicative inverse of $a\pmod{b}$ in the interval $[1,b)$, we have $d=\va$, $c=(a\va -1)/b <a$. We can write
\begin{equation*}
\#  \left\{ \left( \begin{matrix} a & b \\ c & d \end{matrix}\right) \in \cS_Q : \frac{d}{b} \le \beta \right\} =
\sum\limits_{b\le Q} \sum\limits_{\substack{b\le a \le Q \\  d\le \min\{ \beta b,Q-a\} \\ ad \equiv 1 \pmod{b}}} 1
= S_Q^I(\beta)+S_Q^{II}(\beta),
\end{equation*}
where $S_Q^I(\beta)$ denotes the contribution of $b\le Q/(1+\beta)$ and $S_Q^{II}(\beta)$ denotes the contribution of $Q/(1+\beta) < b\le Q$.
Consider the set of points $(x,y)\in \RR^2$ satisfying the inequalities $b\le x\le Q$, $0\le y\le \min\{ \beta b,Q-x\}$.
When $b\le Q/(1+\beta)$ this region is a trapezoid of area 
$A^I_{Q,b}(\beta)=\beta b(2Q-(2+\beta)b)/2$, and when $Q/(1+\beta) < b \le Q$ this region is a right-angled triangle of area
$A^{II}_{Q,b} (\beta)=(Q-b)^2/2$.
Exploiting the Weil-Estermann bound and applying  \cite[Lemma~3.1]{Bo2007} with $T=Q^{1/4}$ to the two regions considered above, we find
\begin{equation}\label{eq3}
 \sum\limits_{\substack{b\le a \le Q \\  d\le \min\{ \beta b,Q-a\} \\ ad \equiv 1 \pmod{b}}} 1 =
\frac{\varphi(b)}{b^2}\, \big( A^I_{Q,b}(\beta)+A^{II}_{Q,b} (\beta)\big) +\cE_b,
\end{equation}
with
\begin{equation*}
\cE_b \ll_\varepsilon Q^{3/4} +Q^{1/4} b^{1/2+\varepsilon} +Qb^{-1/2} \quad \text{and} \quad
\sum\limits_{b\le Q} \cE_b \ll_\varepsilon Q^{7/4+\varepsilon}.
\end{equation*}
Summing over $b$ and plugging in~\eqref{eq3} the formulas for $A^I_{Q,b}(\beta)$ and $A^{II}_{Q,b}(\beta)$, we infer:
\begin{equation*}
\begin{split}
\# \bigg\{\begin{pmatrix} a & b \\ c & d \end{pmatrix}  \in \cS_Q :   \frac{d}{b} \le \beta \bigg\}  & = \frac{\beta}{2}
\sum\limits_{b\le \frac{Q}{1+\beta}} \frac{\varphi(b)}{b}\, \big( 2Q-(2+\beta)b)  +
\sum\limits_{\frac{Q}{1+\beta} < b\le Q} \frac{\varphi(b)}{b^2}\, \frac{(Q-b)^2}{2} +O_\varepsilon (Q^{7/4+\varepsilon}) \\ 
   & =-\frac{\beta(2+\beta)}{2} \sum\limits_{b\leq \frac{Q}{1+\beta}}
   \varphi (b) +\frac{1}{2} 
   \sum\limits_{\frac{Q}{1+\beta}\leq b\leq Q} \varphi (b) + 
   \beta Q \sum\limits_{b\leq \frac{Q}{1+\beta}} 
   \frac{\varphi(b)}{b} \\ 
   & \quad 
   -Q \sum\limits_{\frac{Q}{1+\beta}\leq b\leq Q} 
   \frac{\varphi(b)}{b} +\frac{Q^2}{2} \sum\limits_{\frac{Q}{1+\beta}\leq b\leq Q} \frac{\varphi(b)}{b^2} +O_\varepsilon (Q^{7/4+\varepsilon}).
\end{split}
\end{equation*}

In conjunction with the elementary estimates (see, e.g., \cite[Exercise~5, Chapter~3]{Apo1976})
\begin{equation}\label{eq4}
 \sum\limits_{b\le Q} \varphi (b)=\frac{Q^2}{2\zeta(2)} +O(Q\log Q),\qquad
\sum\limits_{b\le Q} \frac{\varphi(b)}{b} =\frac{Q}{2\zeta(2)}+O(\log Q) ,
\end{equation}
and the less trivial (see, e.g., \cite[Exercise~6, Chapter~3]{Apo1976} and  \cite[Corollary~4.5]{Bo2007} for a proof)
\begin{equation}\label{eq5}
\sum\limits_{b\le Q} \frac{\varphi(b)}{b^2} =\frac{\log Q}{\zeta(2)} +C+O(Q^{-1}\log Q) \quad \text{for some constant $C$,}
\end{equation}
this yields, up to an error of magnitude $O_\varepsilon (Q^{7/4+\varepsilon})$,
\begin{equation*} 
\begin{split}
\# \left\{ \begin{pmatrix} a & b \\ c & d \end{pmatrix} \in \cS_Q : \frac{d}{b} \le \beta \right\} &  = 
\frac{Q^2}{2\zeta(2)} \bigg( -\frac{\beta(2+\beta)}{2(1+\beta)^2} +\frac{1}{2} \Big( 1-\frac{1}{(1+\beta)^2} \Big)
+\frac{\beta}{1+\beta}-1+\frac{1}{1+\beta} \bigg) \\
& 
\quad
+\frac{Q^2}{2\zeta(2)}\bigg(  \log \Big( \frac{Q}{Q/(1+\beta)}\Big)+C-C \bigg)  +O_\varepsilon (Q^{7/4+\varepsilon})\\
&  =\frac{Q^2}{2\zeta(2)} \log (1+\beta) +O_\varepsilon (Q^{7/4+\varepsilon}) ,
\end{split}
\end{equation*}
concluding the proof of~\eqref{eq1} with the weaker error term $O_\varepsilon (Q^{7/4+\varepsilon})$.

\begin{proof}[Proof of Theorem~\ref{theorem1}]
Matrices in $\cS_Q$ can be expressed as
\begin{equation}\label{eq6}
M=\left( \begin{matrix} \vd +kb & b \\  \frac{d\vd -1}{b}+kd & d \end{matrix}\right) \in \operatorname{SL}(2,\ZZ),
\end{equation}
where
\begin{equation*}
\left\{
\begin{aligned}
&d\vd \equiv 1 \pmod{b}, \ 0< \vd <b,\\ 
& 0\leq d\leq b \leq  \vd +kb, \\
& b\le \frac{d\vd -1}{b}+kd   \le \vd+kb ,\\ 
& d+\vd +kb \le Q,\   k\geq 1.
\end{aligned}
\right.
\end{equation*}

While estimating $\# \SF_Q$, we must take $k=1$ in~\eqref{eq6}. 
This entails
\begin{equation*}
b+d+\vd \le Q,
\end{equation*}
so we can write 
\begin{equation*}
\# (\SF_Q \cap [0,\beta]) =\sum\limits_{b\le Q} \sum\limits_{\substack{d \le \beta b \\ 1\le \vd \le b,\,  d+\vd \le Q-b \\ d\vd \equiv 1 \pmod{b}}} 1
=S_Q^I(\beta)+S_Q^{II}(\beta)+S_Q^{III} (\beta)+S_Q^{IV}(\beta) ,
\end{equation*}
where notation is as follows: $S_Q^I(\beta)$ denotes the contribution of $b\le Q/(2+\beta)$, $S_Q^{II}(\beta)$ the contribution of \mbox{$Q/(2+\beta) < b \le Q/2$}, 
$S_Q^{III} (\beta)$ the contribution of $Q/2 <b\le Q/(1+\beta)$, and $S_Q^{IV}(\beta)$ the contribution of $Q/(1+\beta) < b\le Q$.

The set of points $(x,y)\in [0,\beta b]\times [0,b]$, defined by the inequality $x+y \le Q-b$ represents:

\begin{itemize}
\setlength\itemsep{5pt}
\item[(I)] the rectangle$[0,\beta b]\times [0,b]$\\ of area $A^I_{A,b}(\beta)=\beta b^2$ when $b\le Q/(2+\beta)$;
\item[(II)]the pentagon $\{  (x,y): 0\le x \le \beta b, 0\le y\le \min\{ b,Q-b-x\}\}$ \\of area 
$A^{II}_{Q,B} (\beta)= \beta b^2 -((2+\beta) b-Q)^2/2$ when $Q/(2+\beta) \le b \le Q/2$;
\item[(III)] the trapezoid $\{ (x,y):  0\le x \le \beta b , 0\le y \le Q-b-x\}$\\ of area $A^{III}_{Q,b}(\beta)=\beta b(2Q-(2+\beta)b)/2$ when 
$Q/2 \le b \le Q/(1+\beta)$;
\item[(IV)] the triangle $\{ (x,y):  0\le x\le Q- b, 0\le y \le Q-b-x\}$\\ of area $A^{IV}_{Q,b} (\beta) =(Q-b)^2/2$ when $Q/(1+\beta) \le b \le Q$.
\end{itemize}

As in the proof of~\eqref{eq3}, we derive\footnote{With help from \cite[Lemma 2]{Ust2013} one can improve the error from $Q^{7/4+\varepsilon}$ 
to $Q^{3/2+\varepsilon}$. For sake of space we leave the details to the dedicated reader. }
% \begin{equation*}
% \begin{aligned}
% & \# (\SF_Q  \cap [0,\beta]) = \sum\limits_{b\le \frac{Q}{2+\beta}} \frac{\varphi(b)}{b^2} A^{I}_{Q,b} (\beta) +
% \sum\limits_{\frac{Q}{2+\beta} \le b\le \frac{Q}{2}} \frac{\varphi(b)}{b^2} A^{II}_{Q,b}(\beta)   \\
% & \qquad +\sum\limits_{\frac{Q}{2} \le b \le \frac{Q}{1+\beta}} \frac{\varphi(b)}{b^2} A^{III}_{Q,b} (\beta)
% +\sum\limits_{\frac{Q}{1+\beta} \le b\le Q} \frac{\varphi(b)}{b^2}
% A^{IV}_{Q,b} (\beta) +O_\varepsilon (Q^{7/4+\varepsilon}) \\
% & =\beta \sum\limits_{b\le \frac{Q}{2}} \varphi(b)-\frac{1}{2} \sum\limits_{\frac{Q}{2+\beta} \le b \le \frac{Q}{2}} 
% \frac{\varphi(b)}{b^2}\, (3b-Q)\big( (2+\beta)b-Q\big)   \\
% & \qquad  +\frac{\beta}{2} \sum\limits_{\frac{Q}{2} \le b \le \frac{Q}{1+\beta}} \frac{\varphi(b)}{b}\, \big( 2Q-(2+\beta )b\big) +
% \frac{1}{2} \sum\limits_{\frac{Q}{1+\beta} \le b \le Q} \frac{\varphi(b)}{b^2}  \, (Q-b)^2
% +O_\varepsilon (Q^{7/4+\varepsilon}) .
% \end{aligned}
% \end{equation*}
% 
\begin{equation*}
\begin{split}
 \# (\SF_Q  \cap [0,\beta]) &= \sum\limits_{b\le \frac{Q}{2+\beta}} \frac{\varphi(b)}{b^2} A^{I}_{Q,b} (\beta) +
\sum\limits_{\frac{Q}{2+\beta} \le b\le \frac{Q}{2}} \frac{\varphi(b)}{b^2} A^{II}_{Q,b}(\beta)   \\
& \quad +\sum\limits_{\frac{Q}{2} \le b \le \frac{Q}{1+\beta}} \frac{\varphi(b)}{b^2} A^{III}_{Q,b} (\beta)
+\sum\limits_{\frac{Q}{1+\beta} \le b\le Q} \frac{\varphi(b)}{b^2}
A^{IV}_{Q,b} (\beta) +O_\varepsilon (Q^{7/4+\varepsilon}) \\
& =\beta \sum\limits_{b\le \frac{Q}{2}} \varphi(b)-\frac{1}{2} \sum\limits_{\frac{Q}{2+\beta} \le b \le \frac{Q}{2}} 
\frac{\varphi(b)}{b^2}\, \big( (2+\beta)b-Q\big)^2   \\
& \quad  +\frac{\beta}{2} \sum\limits_{\frac{Q}{2} \le b \le \frac{Q}{1+\beta}} \frac{\varphi(b)}{b}\, \big( 2Q-(2+\beta )b\big) +
\frac{1}{2} \sum\limits_{\frac{Q}{1+\beta} \le b \le Q} \frac{\varphi(b)}{b^2}  \, (Q-b)^2
+O_\varepsilon (Q^{7/4+\varepsilon}) .
\end{split}
\end{equation*}

Taking stock on the estimates \eqref{eq4} and \eqref{eq5}, we find that, up
to an error $O(Q\log Q)$, the contribution of sums involving $\varphi (b)$, $\varphi(b)/b$ and $\varphi(b)/b^2$ is
\begin{align*}
     \frac{Q^2}{2\zeta(2)} 
\bigg( \frac{\beta}{4} -\frac{(2+\beta)^2}{2}
\Big( \frac{1}{4} -\frac{1}{(2+\beta)^2}\Big) -
\frac{\beta(2+\beta)}{2}  \Big( \frac{1}{(1+\beta)^2}-\frac{1}{4}
\Big) +\frac{1}{2} \Big( 1-\frac{1}{(1+\beta)^2}\Big) \bigg) =0, 
\end{align*}
\begin{align*}
 \frac{Q^2}{2\zeta(2)} \bigg( (2+\beta)\Big( \frac{1}{2}-\frac{1}{2+\beta}\Big) 
+\beta \Big( \frac{1}{1+\beta}-\frac{1}{2}\Big) -1+\frac{1}{1+\beta}\bigg) =0,
\end{align*}
and respectively
\begin{align*}
-\frac{Q^2}{2\zeta(2)}\, \log \bigg( 
\frac{Q/2}{Q/(2+\beta)}\bigg) +\frac{Q^2}{2\zeta(2)}\, 
\log \bigg( \frac{Q}{Q/(1+\beta)} \bigg) =
\frac{Q^2}{2\zeta(2)} \, 
\log \bigg( \frac{2(1+\beta)}{2+\beta} \bigg) .
\end{align*}

We infer that
\begin{align*}
\# (\SF_Q  \cap [0,\beta])=
\frac{Q^2}{2\zeta(2)} 
\log \bigg( \frac{2(1+\beta)}{2+\beta}\bigg) +O_\varepsilon (Q^{7/4+\varepsilon}),
\end{align*}
which concludes the proof.
\end{proof}

%%%%%%%%%%%%%%%%%%%%%%%%%%%%%%%%%%%%%%%%%%%%%%%%%%%%%%%
\section{\texorpdfstring{The gap distribution of $\SF_Q$}{The gap distribution of cCQ}}\label{Section3}

For each $a/q \in \QQ \cap (0,1)$ with $(a,q)=1$, we set
\begin{equation*} 
h(a/q):=q+a+\va ,
\end{equation*}
where $a\va \equiv 1 \pmod{q}$, $1\le a <q$ if $q\ge 2$. We also set $h(0/1):=1$ and $h(1/1):=3$.
As noticed in the beginning of the proof of Theorem~\ref{theorem1}, we have
\begin{equation*}
\SF_Q =\{ a/q \in \cF_Q  : h(a/q) \le  Q  \} \subseteq \cF_Q.
\end{equation*}

We also introduce the sets
\begin{equation*}
\cH_{Q,r}(\eta):=\left\{
\frac{a_1}{q_1} \in \SF_Q :
\begin{aligned}
\exists\     \frac{a_2}{q_2} ,\ldots ,\frac{a_r}{q_r} \in \cF_Q \setminus \SF_Q,\
 \frac{a_{r+1}}{q_{r+1}} \in \SF_Q ,\  \frac{a_{r+1}}{q_{r+1}} - \frac{a_1}{q_1} \le \frac{\eta}{Q^2}
\\
 \frac{a_1}{q_1} <\frac{a_2}{q_2} < \cdots < \frac{a_{r+1}}{q_{r+1}} \text{ consecutive in $\cF_Q$} 
\end{aligned} 
\right\},\quad r\ge 1, \eta >0.
\end{equation*}

\begin{remark}\label{rem1}
Since $a_{i+1}/q_{i+1}-a_i/q_i =1/(q_i q_{i+1}) > 1/Q^2$, it follows that the set $\cH_{Q,r}(\eta)$ is nonempty only if $1\le r <  \eta$.
\end{remark}

We will prove that for all $r\in \NN$ and $\eta >0$, there exists a constant $C_r(\eta) \in [0,\infty)$ such that 
\begin{equation}\label{eq7}
\# \cH_{Q,r}(\eta) \sim C_r(\eta)Q^2 \quad \text{as $Q\rightarrow \infty$}.
\end{equation}
Since $\# \SF_Q \sim A Q^2$ with $A=\log(4/3)/(2\zeta(2))$, this will subsequently imply the existence of the
limiting gap distribution of $(\SF_Q)_Q$. More precisely, we have $\Delta_{\rm av}=1/(N(Q)-1) \sim 1/(AQ^2)$ and
\begin{equation}\label{eq8}
\cG (\lambda):=\lim\limits_Q \frac{\#\{ \gamma_1 <\gamma_2 \text{ consecutive in $\SF_Q$}: 
\gamma_2 -\gamma_1 \le \lambda \Delta_{\rm av}\}}{N(Q) -1} 
=\frac{1}{A} \sum\limits_{1\le r <\lambda/A} C_r (\lambda/A), \text{ for } \lambda >0 .
\end{equation}

According to Remark~\ref{rem1}, we also have $\# \cH_{Q,r}(\eta)=0$ whenever $r\ge 2$ and $\eta < 2$, hence
\begin{equation}\label{eq9}
\cG (\lambda ) = \frac{C_1 (\lambda/A)}{A} \text{ for } \lambda \in (A,2A].
\end{equation}

%%%%%%%%%%%%%%%%%%%%%%%%%%%%%%%%%%%%%%%%%
\subsection{\texorpdfstring{Work on $\cH_{Q,1}(\eta)$}
{Work on HQ(eta) in terms of consecutive denominators}}\label{Sect3.1}
Firstly, we are interested in estimating the cardinality of the set
\begin{equation*}
    \cH_{Q,1}(\eta)
     =     \left\{  \frac{a_1}{q_1} \in \SF_{Q}: \ 
   \exists\  \frac{a_2}{q_2} \in \SF_Q,\  \frac{a_1}{q_1} <\frac{a_2}{q_2} \text{ consecutive in } \cF_{Q}, \
\frac{1}{q_1 q_2} \le \frac{\eta}{Q^2}
    \right\},
\end{equation*}
which is the same as 
\begin{equation*}
    \#   \left\{ (q_1,q_2,a_1,a_2)\in \NN^4 :
    \begin{tabular}{l}
             $a_i\le q_i \le Q, \  q_1+q_2 >Q, \
             q_1 q_2 \ge Q^2 /\eta$ 
            \\[6pt] 
    $a_2 q_1 -a_1 q_2  =1, \ h(a_1/q_1) \le Q, \  h(a_2/q_2) \le Q$
    \end{tabular}     
    \right\}.
\end{equation*}

The following elementary  lemma will be useful.

\begin{lemma}\label{lem5}
Suppose that $a_1/q_1<a_2/q_2$, $0<a_1<q_1$, $0< a_2 < q_2$, and 
\begin{equation*}
    a_2 q_1 -a_1 q_2 =1 .
\end{equation*}
 Denote by $\va_1$, respectively $\vq_2$, the multiplicative inverse of $a_1 \pmod{q_1}$,
 respectively $q_2 \pmod{q_1}$,
in~$[1,q_1)$ and by $\vva_2$, respectively $\vvq_1$, the multiplicative inverse of 
$a_2\pmod{q_2}$, respectively $q_1\pmod{q_2}$, in $[1,q_2)$. Then 
\begin{itemize}
 \item[$(i)$]
 $\displaystyle \va_1 =\bigg( 1+\bigg\lfloor \frac{q_2}{q_1}\bigg\rfloor \bigg) q_1 -q_2 ,
 \quad \vva_2 =q_1 -\bigg\lfloor \frac{q_1}{q_2}\bigg\rfloor q_2 ,$
 \item[$(ii)$] $\displaystyle a_1 =q_1 -\vq_2,\quad 
 a_2 =\vvq_1 ,\quad a_1 =\frac{q_1 \vvq_1 -1}{q_2} ,
 \quad a_2 =q_2 -\frac{q_2 \vq_2 -1}{q_1}.$
 \item[$(iii)$]
$\displaystyle h(a_1/q_1)  =\bigg( 2+\bigg\lfloor \frac{q_2}{q_1}\bigg\rfloor \bigg) 
q_1 -q_2+\frac{q_1 \vvq_1 -1}{q_2} ,$
\item[$(iv)$] 
$\displaystyle h(a_2/q_2)= q_1 +
        \bigg( 2-\bigg\lfloor \frac{q_1}{q_2}\bigg\rfloor \bigg) q_2 
        -\frac{q_2 \vq_2 -1}{q_1} .$
\end{itemize}
\end{lemma}

\begin{proof}
(i)  Note that $q_1/q_2, q_2/q_1 \notin \ZZ$. The equality $a_2 q_1-a_1 q_2=1$ yields $\va_1 =-q_2+kq_1$, $\vva_2 =q_1 -\ell q_2$, 
for some $k,\ell \in \ZZ$. Imposing the constraints $0<\va_1<q_1$ and $0<\va_2 <q_2$, we find
$k=1+\lfloor q_2/q_1 \rfloor$ and $\ell=\lfloor q_1 /q_2 \rfloor$.

(ii) follows from $\vq_2=-a_1+rq_1$, $\vvq_1=a_2+sq_2$, with $r,s\in \ZZ$, where $0<\vq_2 <q_1$, $0<\vvq_1 <q_2$
yield $r=1$ and $s=0$.

(iii) and (iv) follow from (i), (ii) and the definition of $h$.
\end{proof}

\begin{corollary}\label{cor6}
Suppose that $a_1/q_1<a_2/q_2$ are consecutive elements in $\cF_Q$. 
\begin{enumerate}[itemsep=4pt,label=\upshape(\emph{\roman*})]
%leftmargin=20px,
% \renewcommand{\theenumi}{\roman{enumi}}
\item[\namedlabel{cor61}{\normalfont{(\textit{i})}}]
% \item[$(i)$]
If $q_1<q_2$, then $h(a_2/q_2) >Q$ (that is, $a_2/q_2 \in\SF_Q$ implies $q_1 \ge q_2$).
\item[\namedlabel{cor62}{\normalfont{(\textit{ii})}}]
% \item[$(ii)$]
If $q_2 <q_1$ and $h(a_1/q_1) \le Q$, then $h(a_2/q_2) \le Q$.
\end{enumerate}
\end{corollary}

\begin{proof}
  Part~\ref{cor61} follows from $h(a_2/q_2) =q_2+a_2 +\vva_2 = q_2 +a_2 +q_1 > q_2 + q_1  >Q$.

For~\ref{cor62},  Lemma~\ref{lem5} provides:
\begin{align*}
   h(a_1/q_1)=q_1+a_1+\va_1 &=3q_1 -q_2 -\vq_2\\ 
\intertext{and}
   h(a_2/q_2) = q_2+a_2+\vva_2&=q_1+(2-\lfloor q_1/q_2 \rfloor )q_2-  (q_2 \vq_2 -1)/q_1.
\end{align*}
% $q_1+a_1+\va_1 =3q_1 -q_2 -\vq_2$ and 
% $q_2+a_2+\vva_2=q_1+(2-\lfloor q_1/q_2 \rfloor )q_2-  (q_2 \vq_2 -1)/q_1$.

When $k:=\lfloor q_1/q_2\rfloor \geq 2$, the conclusion follows from
$q_2 \vq_2/q_1 \ge  1/q_1 \geq (2-k)q_2 +q_1-Q+1/q_1$.
When $k=1$, note that
$q_1 (1-(Q-q_1)/q_2) \le q_1 (1- (Q-q_1)/q_1) =2q_1 -Q < 3q_1 -Q-q_2$.
Since $Q-q_1 <q_2<q_1$ and $q_2+\vq_2 \ge 3q_1-Q$, we get
\begin{equation*}
    \frac{q_2 \vq_2}{q_1} \ge \frac{q_2}{q_1} (3q_1 -Q-q_2) \geq
\frac{q_2}{q_1} \bigg( 1+q_1 \Big( 1-\frac{Q-q_1}{q_2}\Big) \bigg)
    \geq  q_1+q_2-Q +\frac{1}{q_1},
\end{equation*}
as desired.
    \end{proof}

Returning to $\cH_{Q,1}(\eta)$, we consider pairs $a_1/q_1 <a_2/q_2$ of fractions from $\SF_Q$ which are also consecutive in $\cF_Q$. 
Corollary~\ref{cor6}.\ref{cor61} gives \fbox{$q_2 <q_1$}, so  $q:=q_1\in [Q/2,Q]$. Next,  Lemma~\ref{lem5} 
yields $q_1 +a_1 +\va_1=3q_1  - q_2-\vq_2$, and so
\begin{equation*}
\# \cH_{Q,1} (\eta)  =\sum\limits_{\frac{Q}{2} \le  q\leq Q} \# \left\{ (\alpha,\beta): 
\begin{tabular}{l} 
$\alpha \beta \equiv 1  \pmod{q},\  \max\{ Q-q, Q^2/(\eta q)\} \le \alpha \le q$ \\ 
$\max\{ 3q-Q-\alpha ,0\} \le  \beta \le q$ \end{tabular} \right\} .
\end{equation*}

\begin{remark}\label{rem2}
The inequality $q_1 q_2 \ge Q^2/\eta$ yields $\min\{ q_1,q_2\} \ge Q/\eta$.
\end{remark}

Proceeding as in the proof of~\eqref{eq3}, then applying M\" obius summation (cf., e.g. \cite[Lemma~2.3]{BCZ2000}) and applying
the changes of variable $( x,y)=(Qu,Qv)$ and $q=Qw$, we find
\begin{equation}\label{eq10}
\begin{aligned}
\# \cH_{Q,1} (\eta)  & = \sum\limits_{\frac{Q}{2} \le  q\le Q} \frac{\varphi(q)}{q^2} \operatorname{Area} \left\{  (x,y) :
\begin{tabular}{l}
$\max\{ Q-q,Q^2 /(\eta q)\} \le x \le q$  \\   $\max\{3q-Q-x,0\} \le y \le  q$ 
\end{tabular}
\right\}  +O_\varepsilon (Q^{7/4+\varepsilon})  \\ 
& = \frac{Q^2}{\zeta(2)}  \int_{\frac{1}{2}}^1 \frac{\dd w}{w} \operatorname{Area} \Omega_{1} (w,\eta) +O_\varepsilon (Q^{7/4+\varepsilon}), \quad \text{where} \\
\Omega_{1}  (w,\eta) &  =    \bigg\{ (u,v): 
\max\Big\{ 1-w,\frac{1}{\eta w}\Big\} \le u\le w, \  \max\{ 3w-1-u, 0 \} \le v\le w\bigg\}.  
\end{aligned}
\end{equation} 

The main term in~\eqref{eq10} can be readily computed for every $\eta >0$. For instance, it is equal to $0$ when $0<\eta \le 1$ as $w\le 1/(\eta w)$ for every $w\in (0,1]$.

Further, %\spaceskip=3pt plus 2pt minus 0pt
when $1<\eta \le 2$, we have $2/3 < 1/\sqrt{\eta} \le w_0(\eta):=(1+\sqrt{1+8/\eta})/4 < 1$.
When \mbox{$1/2 \le w \le 1/\sqrt{\eta}$}, we have 
$\Area \Omega_1 (w,\eta)=0$.
When $1/\sqrt{\eta} \le w \le w_0(\eta)$, we have 
\mbox{$2w-1 \le 1/(\eta w)\le w$} and 
$\Area \Omega_1 (w,\eta)= (w-1/(\eta w))(2-3w+1/(\eta w))/2$.
When \mbox{$w_0(\eta) \le w \le 1$}, we have $1/(\eta w) \le 2w-1 \le w$ and
$\Area \Omega_1 (w,\eta)=(1-w)^2/2$. 
This yields
\begin{equation*}
C_1 (\eta) =\frac{1}{2\zeta(2)} \int_{\frac{1}{\sqrt{\eta}}}^{w_0(\eta)} \bigg( w-\frac{1}{\eta w}\bigg)\bigg( 2-3w+\frac{1}{\eta w}\bigg)\frac{dw}{w}
+ \frac{1}{2\zeta(2)} \int_{w_0(\eta)}^1 \frac{(1-w)^2}{2w}\, dw , \text{ for } \eta \in [1,2].
\end{equation*}
In conjunction with~\eqref{eq9} this provides an explicit formula for $\cG(\lambda)$ when
$\lambda\in [0,2A]$.

On the other hand, when $\eta \ge 9/2$ and $2/3 \le w \le 1$ we have $1/\eta \le 2/9 \le w(2w-1)$, hence $1/(\eta w) \le 2w-1$ and
$\Area \Omega_1 (w,\eta) =(1-w)^2 /2$. When $1/2\le w\le 2/3$, we have $1/\eta \le  1/4\le w(1-w)$,
hence $1/(\eta w) \le 1-w$ and  $\Area \Omega_1 (w,\eta) =(3-4w)(2w-1)/2$. We infer
\begin{equation*}
\begin{aligned}
C_1 (\eta) & =\frac{1}{\zeta(2)} \int_{\frac{2}{3}}^1 \frac{(1-w)^2}{2w}\, dw + \frac{1}{\zeta(2)} 
\int_{\frac{1}{2}}^{\frac{2}{3}} \frac{(3-4w)(2w-1)}{2w}\, dw \\
& =\frac{2\log 3  -(7/2)\log 2 +1/4}{\zeta(2)} ,
\end{aligned}
\end{equation*}
so in particular $C_1(\eta)$ does not depend on $\eta$ when $\eta \ge 9/2$.

%%%%%%%%%%%%%%%%%%%%%%%%%%%%%%%%%%%%%%%%%%%%%

\subsection{\texorpdfstring{Work on $\cH_{Q,2}(\eta)$}
{The expression of HQ2(eta) in terms of consecutive denominators}}
Next, we analyze the contribution of triples $a_1/q_1 < a_2/q_2 < a_3/q_3$
consisting of consecutive elements in $\cF_Q$ satisfying
conditions $a_1/q_1\in \SF_Q$,
$a_2/q_2 \notin \SF_Q$, \mbox{$a_3/q_3 \in \SF_Q$} and 
$a_3/q_3 - a_1/q_1= 1/(q_1 q_2)+1/(q_2 q_3) \le \eta/Q^2$.

The sets
$\cH_{Q,2} (\eta)$ and
\begin{equation*}
   \widetilde{ \cH}_{Q,2}(\eta) := \left\{ (q_1,q_2,q_3,a_1,a_2,a_3)\in \ZZ^6 :
    \begin{tabular}{l}
    $a_2 q_1 - a_1 q_2 = a_3 q_2 - a_2 q_3 = 1,\  1\le a_i< q_i \le Q$ \\
    $q_1 + q_2 > Q, \  q_2+q_3 >Q,\  1/(q_1 q_2) + 1/(q_2 q_3) \le \eta/Q^2$  \\
    $h(a_1/q_1) \le Q,\  h(a_2/q_2) >Q,\  h(a_3/q_3) \le Q$
    \end{tabular}
    \right\}
    \end{equation*}
are in one-to-one correspondence.
Since $a_1/q_1 < a_2/q_2$ are consecutive in $\cF_Q$ and $h(a_1/q_1) \le Q$, 
if $q_2 <q_1$ then Corollary~\ref{cor6}.\ref{cor62} entails $h(a_2/q_2) \le Q$, which is not allowed.
Hence in the definition of $\widetilde{\cH}_{Q,2}(\eta)$ above, we can always assume that 
\fbox{$q_1<q_2$}, and also remove the inequality $h(a_2/q_2) > Q$.
Note also that Remark~\ref{rem2} entails \fbox{$q_2 < \eta q_1$}.

For every integer $b$, denote by $\bar{b}$, $\bar{\bar{b}}$, respectively $\bar{\bar{\bar{b}}}$ the 
multiplicative inverse of $b\pmod{q_1}$ in $[1,q_1)$, $b\pmod{q_2}$ in $[1,q_2)$,
and respectively of $b\pmod{q_3}$ in $[1,q_3)$.
From Lemma~\ref{lem5} we infer $a_2 =\vvq_1$, 
$a_1 =(q_1\vvq_1 -1)/q_2$ and $\va_1 =-q_2+(1+\lfloor q_2/q_1 \rfloor ) q_1$, so that
\begin{equation}\label{eq11}
    h(a_1/q_1) \le Q \ \Longleftrightarrow \  \bigg( 2+\bigg\lfloor
    \frac{q_2}{q_1} \bigg\rfloor \bigg) q_1 -q_2 +\frac{q_1 \vvq_1 -1}{q_2} \le Q.
\end{equation}

From $a_3 q_2 -a_2 q_3 =1$ we infer $0< \vvva_3 =q_2 - \lfloor q_2/q_3 \rfloor  q_3 <q_3$.

On the other hand, letting $\nu:=\lfloor (Q+q_1)/q_2 \rfloor$, we have
$q_3=\nu q_2-q_1$, $a_3 =\nu a_2-a_1 $\xspace$=\nu \vvq_1 - (q_1 \vvq_1 -1)/q_2$, and
$1\le \nu < (Q+q_2)/q_2 <3$.

However, it is easily seen that $\nu=2$ cannot occur. Indeed, $\nu=2$ would imply 
\mbox{$q_3=2q_2-q_1 >q_2$}, $\vvva_3=q_2$, 
$a_3 =2a_2-a_1 =2\vvq_1 -(q_1 \vvq_1 -1)/q_2$, so that
\begin{equation}\label{eq12}
h(a_3/q_3)=
    q_3 +a_3 +\vvva_3 =
    3q_2 -q_1 +2\vvq_1 -\frac{q_1 \vvq_1 -1}{q_2} \le Q.
\end{equation}
On the other hand, in this situation we also have $\lfloor q_2/q_1 \rfloor =1$, so~\eqref{eq11} becomes
\begin{equation}\label{eq13}
 3 q_1 -q_2 +\frac{q_1 \vvq_1 -1}{q_2} \le Q.
\end{equation}
Adding~\eqref{eq12} and~\eqref{eq13} we get $2q_1 +2q_2 +2\vvq_1 \le 2Q$, which contradicts $q_1 +q_2 >Q$.

We are only left with $\nu=1$, which entails $q_3=q_2-q_1 <q_2$ and $a_3 =a_2 -a_1=\vvq_1 - (q_1 \vvq_1 -1)/q_2$,
so that $2Q/3 \le q_2 \le Q$, 
$1/(q_1q_2)+1/(q_2q_3)=1/(q_1q_3)=1/(q_1(q_2-q_1))$, and
\begin{equation*}
    q_3 +a_3 +\vvva_3 \le Q \ \Longleftrightarrow \
    \bigg(  2-\bigg\lfloor \frac{q_2}{q_2 -q_1} \bigg\rfloor \bigg) q_2 
    +\bigg(\bigg\lfloor \frac{q_2}{q_2 -q_1} \bigg\rfloor -1\bigg) q_1 +\vvq_1 
    -\frac{q_1 \vvq_1 -1}{q_2} \le Q .
\end{equation*}
Since $q_1+q_2 >Q$, $q_1 <q_2$ and $\lfloor (Q+q_1)/q_2 \rfloor =1$, we have $q_2 >Q/2$ and
\begin{equation*}
(q_1,q_2) \in Q\cT_1 :=\{ (\alpha,q) : q+\alpha >Q, 0< \alpha <2q-Q\le Q \} .
\end{equation*}
We partition the triangle $Q\cT_1$  into the triangles:
\begin{equation*}
\begin{split}
        \cT_Q^{(1)} & := \{ (\alpha,q)\in Q\cT_+ : Q-q <\alpha <q/2 \} \quad \text{ and } \\
        \cT_Q^{(2)}  & :=  \{ (\alpha,q) \in Q \cT_+ ,\  q/2 <\alpha <2q-Q \},
\end{split}
\end{equation*}
and analyze these two situations:

\texttt{Case 1.} $Q-q <\alpha < q/2$. Then $\nu =1$,
$\lfloor q/\alpha \rfloor \ge 2$ and $\lfloor q/(q-\alpha) \rfloor =
\lfloor 1/(1-\alpha/q) \rfloor =1$. 

Let $k:=1+\lfloor  q/\alpha \rfloor$, so $q/k < \alpha \le  q/(k-1)$. Note that \fbox{$3\le  k<\eta$}.
The triangle $\cT_Q^{(1)}$ is in turn partitioned  into the quadrangles 
\begin{equation*}
\cQ_Q^{(k)}:=    \left\{ (\alpha,q)\in \cT_Q^{(1)} : (k-1)\alpha \le  q < k\alpha \right\} 
    \subseteq \bigg[ \frac{Q}{k+1},\frac{Q}{k-1}\bigg] \times 
    \bigg[ \frac{(k-1)Q}{k},Q\bigg],\quad k\ge 3,
\end{equation*}
with contribution to $\widetilde{\cH}_{Q,2}(\eta)$ given by
\begin{equation*}
S_Q^{(1)} (\eta)= \sum\limits_{3\leq k<1+\eta} \sum\limits_{\frac{(k-1)Q}{k}< q \le Q} A_k (q),
\end{equation*}
where
\begin{equation*}
    A_k (q) =  \#
    \left\{ (\alpha,\beta) :  \begin{aligned}
    & \alpha \beta \equiv 1 \pmod{q},\  \max \bigg\{ Q-q,\frac{q}{k} \bigg\} < \alpha 
    \le \frac{q}{k-1},  \  \alpha (q-\alpha) \ge \frac{Q^2 }{\eta} \\  & 0<\beta \le q,\
    (k+1)\alpha -q +\frac{\alpha\beta -1}{q} \le Q,\
    q +\beta -\frac{\alpha \beta -1}{q} \le Q
    \end{aligned}\right\}.
\end{equation*}

An application of  \cite[Lemma 3.1]{Bo2007} with $T=Q^{1/4}$  to the region
\begin{equation*}
    \Omega_Q^{(k)}  (q,\eta) := \left\{
    (\alpha,\beta): \begin{aligned}
& \max \bigg\{  Q-q,\frac{q}{k}  \bigg\} < \alpha \le \frac{q}{k-1} ,\  \alpha(q-\alpha) \ge \frac{Q^2}{\eta} \\
&  0 \le \beta \le \min\bigg\{ q, \frac{q(Q-q)}{q-\alpha}, \frac{q(Q+q-(k+1)\alpha)}{\alpha}\bigg\}
    \end{aligned}
    \right\}
\end{equation*}
leads to 
\begin{equation}\label{eq14}
   A_k (q) = \frac{\varphi (q)}{q^2} \Area  \Omega_Q^{(k)} (q,\eta)
   +\cE_{k,q},
\end{equation}
with
\begin{equation*}
    \cE_{k.q} \ll_{k,\varepsilon} Q^{-1/4} q+Q^{1/4} q^{1/2+\varepsilon}  
\end{equation*}
and
\begin{equation}\label{eq15}
  \sum\limits_{3\le k<1+\eta}  \sum\limits_{\frac{(k-1)Q}{k} < q \le Q} \cE_{k,q} \ll_{\varepsilon,\eta} 
Q^{7/4+\varepsilon}.
\end{equation}
Combining~\eqref{eq14} and~\eqref{eq15} we infer
\begin{equation*}
    S_Q^{(1)} (\eta) = \sum\limits_{3\le k<1+\eta } \sum\limits_{\frac{(k--1)Q}{k} < q \le Q}
    \frac{\varphi(q)}{q^2} \Area \Omega_Q^{(k)} (q,\eta) +
    O_{\varepsilon,\eta} (Q^{7/4+\varepsilon}) .
\end{equation*}

Applying M\" obius summation to the sum in $q$ above as in Section \ref{Sect3.1}, we deduce
\begin{equation}\label{eq16}
\begin{aligned}
S_Q^{(1)}(\eta) &  = \frac{Q^2}{\zeta (2)} \sum\limits_{3\le k<1+\eta} \int_{1-\frac{1}{k}}^1 \frac{dw}{w}
\Area \Omega^{(k)}_{2,1} (w,\eta)
+ O_{\varepsilon,\eta} (Q^{7/4+\varepsilon}), \quad \text{where} \\
   \Omega^{(k)}_{2,1}  (w,\eta) & = \left\{ (u,v) :
    \begin{aligned}
    &    \max \bigg\{ 1-w,\frac{w}{k} \bigg\} \le u \le  \frac{w}{k-1},\  u(w-u)\ge \frac{1}{\eta}  \\
   & 0\le v \le \min\bigg\{ w, \frac{w(1-w)}{w-u}, \frac{w(1+w-(k+1)u)}{u}\bigg\}
    \end{aligned}
    \right\} .
\end{aligned}
\end{equation}

\texttt{Case 2.} $q/2 <\alpha < 2q-Q$. Then $\nu=1$ and
$\lfloor q/\alpha \rfloor =1$. Set  $\ell:= \lfloor q/(q-\alpha) \rfloor$, so
$(\ell-1)q/\ell  \le \alpha < \ell q/(\ell+1)$. 
We have $2\le \ell=\lfloor q_2/q_3 \rfloor <\eta$.
The triangle $\cT_Q^{(2)}$ is partitioned  into the quadrangles 
\begin{equation*}
    \bigg\{ (\alpha,q)\in \cT_Q^{(2)} : \frac{(\ell +1)\alpha}{\ell} 
    <  q \le   \frac{\ell \alpha}{\ell -1} 
    \bigg\} \subseteq \bigg[ \frac{(\ell -1) Q}{2\ell -1},\frac{\ell Q}{\ell +1}\bigg] \times 
    \bigg[ \frac{(\ell +1)Q}{2\ell +1},Q\bigg], \quad \ell \ge 2,
\end{equation*}
with contribution to $\widetilde{\cH}_{Q,2}(\eta)$ given by
\begin{equation*}
    S_Q^{(2)} (\eta) = \sum\limits_{\substack{2\le \ell <\eta \\  \frac{(\ell +1)Q}{2\ell +1}< q \le Q}} \#
    \left\{ (\alpha,\beta) :\
    \begin{aligned}
  &   \alpha \beta \equiv 1 \pmod{q},\  \max \bigg\{ Q-q,\frac{(\ell -1)q}{\ell} \bigg\} \le \alpha 
    \le \frac{\ell q}{\ell +1}    \\
&   \alpha (q-\alpha) \ge \frac{Q^2 }{\eta}, \\
&   0 \le \beta \le \min\bigg\{ q,\frac{q(Q+q-3\alpha)}{\alpha}, \frac{q(Q+(\ell-2)q-(\ell-1)\alpha}{q-\alpha} \bigg\} 
    \end{aligned}\right\}.
\end{equation*}

Proceeding as in \texttt{Case 1}, we obtain
\begin{equation}\label{eq17}
\begin{aligned}
S_Q^{(2)} (\eta) & =\frac{Q^2}{\zeta(2)} \sum\limits_{2\le   \ell < \eta} \int_{\frac{\ell+1}{2\ell+1}}^1 \frac{dw}{w} \Area 
\Omega^{(\ell)}_{2,2} (w,\eta) +O_{\varepsilon,\eta} (Q^{7/4+\varepsilon}), \quad \text{where} \\
\Omega^{(\ell)}_{2,2} (w,\eta) & = \left\{ (u,v) :
    \begin{aligned}
        & \max \bigg\{ 1-w,\frac{(\ell-1)w}{\ell} \bigg\} \le u \le  \frac{\ell w}{\ell +1},\  u(w-u)\ge \frac{1}{\eta}   \\
   & 0\le v \le \min \bigg\{  w, \frac{w(1+w-3u)}{u}, \frac{w(1+(\ell -2)w-(\ell-1)u)}{w-u}  \bigg\}
    \end{aligned}
    \right\}  .
\end{aligned}
\end{equation}

%%%%%%%%%%%%%%%%%%%%%%%%%%%%%%%%%%%%%%%%%

\subsection{\texorpdfstring{Work on $\cH_{Q,r}(\eta)$, $r\geq 3$}
{Work on HQr(eta) whenrgeq3}}
Here, we count the number of $(r+1)$-tuples $\gamma_1 < \gamma_2 < \cdots < \gamma_{r+1}$ of consecutive elements in $\cF_Q$, $\gamma_i =a_i/q_i$, with 
$\gamma_1, \gamma_{r+1} \in \SF_Q$, 
$\gamma_2,\ldots \gamma_r \notin \SF_Q$, and
\begin{equation*}
\gamma_{r+1} - \gamma_1 =\sum\limits_{i=1}^r \frac{1}{q_i q_{i+1}} 
\le \frac{\eta}{Q^2} .
\end{equation*}

The set $\cH_{Q,r} (\eta)$ has the same cardinality as the set of all integer tuples 
$(q_1,\ldots,q_{r+1},a_1,\ldots,a_{r+1})$ which satisfy
\begin{equation*}
\left\{
\begin{aligned}
& a_2 q_1-a_1 q_2 =a_3 q_2-a_2 q_3  = \cdots =a_{r+1} q_r -a_r q_{r+1} =1 ,\quad
1\le a_i < q_i \le Q,\  i=1,\ldots, r+1, \\
&  q_i +q_{i+1}>Q, \quad \max\{ h(a_1/q_1),h(a_{r+1}/q_{r+1})\} \le Q 
< \min\{ h(a_2/q_2),\ldots,h(a_r/q_r)\} ,\\
& \sum\limits_{i=1}^r \frac{1}{q_i q_{i+1}} \le \frac{\eta}{Q^2}.
\end{aligned}
\right.
\end{equation*}
Note that Corollary \ref{cor6} entails $q_2>q_1$.

Since $a_1/q_1 < \cdots < a_{r+1}/q_{r+1}$ are consecutive elements in $\cF_Q$,
the pair $(q_1,q_2)$ uniquely determines $a_1$, $a_2$, and all denominators $q_3,\ldots,q_{r+1}$. The latter can be expressed 
employing  the iterates $T^i=(L_{i},L_{i+1})$ of the map $T$ from~\eqref{eq2}.
More precisely, we have
\begin{equation*}
\Lambda_{i-1}:=\frac{q_i}{Q} =L_{i-1} \bigg( \frac{q_1}{Q},\frac{q_2}{Q}\bigg) ,\quad \frac{q_{i+1}}{Q}=L_i \bigg( \frac{q_1}{Q},\frac{q_2}{Q}\bigg) ,\quad i\ge 1 .
\end{equation*}

To take advantage of the Weil-Estermann bound, it is more convenient to count triples $(q,\alpha,\beta)$ with $\alpha\beta \equiv 1 \pmod{q}$ instead of
pairs $(q_1,q_2)$, where $q=q_2$, $\alpha=q_1$ and $\beta =\vvq_1 =a_2$. 
More precisely, taking stock on Lemma~\ref{lem5}, we see that  $\# \cH_{Q,r}(\eta)$ 
is given by the number of positive  integer triples $(q,\alpha,\beta)$ which satisfy 
\begin{equation}\label{eq18}
\left\{
\begin{aligned}
&\frac{Q}{2} < q \le  Q, \quad \alpha \beta \equiv 1 \pmod{q},\quad  Q-q  \le \alpha \le q,\quad  0 \le \beta \le q, \\
& \bigg( 2+\bigg\lfloor \frac{q}{\alpha} \bigg\rfloor \bigg)\alpha -q +\frac{\alpha \beta -1}{q} \le Q, \\
& \bigg(2+\bigg\lfloor \frac{q_{i+1}}{q_i}\bigg\rfloor \bigg)q_i -q_{i+1} +a_i >Q, \quad   i=2,\ldots ,r,  \\
& \bigg(1-\bigg\lfloor \frac{q_r}{q_{r+1}} \bigg\rfloor \bigg) q_{r+1} +q_r +a_{r+1} \le Q, \\
&\sum\limits_{i=1}^r \frac{1}{q_i q_{i+1}} \le  \frac{\eta}{Q^2} .
\end{aligned}
\right.
\end{equation}

We also need to introduce the \emph{Farey continuants}, that is, the noncommutative polynomials defined in \cite[Sect.~3.2]{BHS2013} (see also \cite{CZ2015}
and \cite{BH2011}), by $K^F_{-1}(\cdot) =0$, $K^F_0 (\cdot)=1$, $K^F_1 (x)=x$, and
\begin{equation*}
K^F_\ell (x_1,\ldots,x_\ell) =x_\ell K^F_{\ell -1} (x_1,\ldots,x_{\ell-1}) -K^F_{\ell-2} (x_1,\ldots,x_{\ell-2}) .
\end{equation*}
Setting $\nu_i :=\kappa (q_{i-1}/Q,q_i/Q)$ we get
\begin{equation*}
\begin{split} 
(q_{i+1},q_i) & = (q_2,q_1) \left( \begin{matrix} \nu_2 & 1 \\ -1 & 0 \end{matrix}\right) \cdots 
\left( \begin{matrix} \nu_i & 1 \\ -1 & 0 \end{matrix} \right)\\
& = (q_2,q_1) \left( \begin{matrix} K^F_{i-1} (\nu_2,\ldots,\nu_i )  &  K^F_{i-2} (\nu_2 ,\ldots, \nu_{i-1}) \\  
- K^F_{i-2} (\nu_3,\ldots,\nu_i) &  - K^F_{i-3} (\nu_3,\ldots,\nu_{i-1})\end{matrix}\right)
\end{split}
\end{equation*}
and
\begin{equation*}
(a_{i+1},a_i)  = (a_2,a_1) \left( \begin{matrix} K^F_{i-1} (\nu_2,\ldots,\nu_i )  &  K^F_{i-2} (\nu_2 ,\ldots, \nu_{i-1}) \\  
- K^F_{i-2} (\nu_3,\ldots,\nu_i) &  - K^F_{i-3} (\nu_3,\ldots,\nu_{i-1})\end{matrix}\right).
\end{equation*}
The previous equality yields\footnote{Note that $\nu_2,\ldots,\nu_i$ depend only on $q$, $\alpha$ and $\beta$ via the iterates of the map $T$.}
\begin{equation*}
\begin{split}
a_i  & =  K^F_{i-2} ( \nu_2 ,\ldots , \nu_{i-1}) a_2  -K^F_{i-3} ( \nu_3,\ldots , \nu_{i-1}) a_1  \\
&  = K^F_{i-2}(\nu_2,\ldots,\nu_{i-1}) \beta  -K^F_{i-3} (\nu_3,\ldots,\nu_{i-1}) \frac{\alpha\beta-1}{q} \\
& = \frac{q_i \beta +K^F_{i-3} (\nu_3,\ldots,\nu_{i-1})}{q} = \frac{Q\Lambda_{i-1}\beta +K^F_{i-3}(\nu_3,\ldots,\nu_{i-1})}{q} .
\end{split}
\end{equation*} 
Then conditions~\eqref{eq18} become
\begin{equation}\label{eq19}
\left\{ \begin{aligned}
& \frac{Q}{2} < q \le  Q, \quad \alpha \beta \equiv 1 \pmod{q},\quad  Q-q  \le \alpha \le q 
,\quad  0\le \beta \le q, \\
&  \bigg( 2+\bigg\lfloor \frac{q}{\alpha} \bigg\rfloor \bigg)\alpha -q +\frac{\alpha \beta -1}{q} \le Q, \\
& \bigg(2+\bigg\lfloor \frac{\Lambda_i}{\Lambda_{i-1}}\bigg\rfloor \bigg) Q\Lambda_{i-1} -Q\Lambda_i  
 + \frac{Q\Lambda_{i-1} \beta+K^F_{i-3} (\nu_3,\ldots,\nu_{i-1})}{q} >Q, \quad   i=2,\ldots,r, \\
& \bigg(1-\bigg\lfloor \frac{\Lambda_{r-1}}{\Lambda_r} \bigg\rfloor \bigg) Q \Lambda_r + Q\Lambda_{r-1} 
+ \frac{Q\Lambda_r \beta +K^F_{r-2} (\nu_3,\ldots,\nu_r)}{q}  \le Q , \\
& \sum\limits_{i=1}^{r} \frac{1}{\Lambda_{i-1} \Lambda_i} \le \eta .
\end{aligned}
\right.
\end{equation}

When $q_i >1$ we have 
\begin{equation*}
\bigg\lfloor \frac{q_{i+1}}{q_i}\bigg\rfloor +\bigg\lfloor \frac{q_{i-1}}{q_i}\bigg\rfloor =
\kappa \bigg( \frac{q_{i-1}}{Q},\frac{q_i}{Q}\bigg)  -1,
\end{equation*}
which entails 
\begin{equation*}
\max\limits_{2\le i\le r}\bigg\{  \bigg\lfloor \frac{q_{i+1}}{q_i}\bigg\rfloor ,\bigg\lfloor \frac{q_{i-1}}{q_i} \bigg\rfloor \bigg\} \le \eta -1 .
\end{equation*}
Inequalities $1/(q_i q_{i+1}) \le \eta/Q^2$, $i=1,\ldots,r$, also yield 
\begin{equation}\label{eq20}
\min \{ q_1,q_2,\ldots,q_r\} \ge \frac{Q}{\eta}\quad \text{and} \quad
\nu_i =\bigg\lfloor \frac{Q+q_{i-1}}{q_{i}}\bigg\rfloor \le 2\eta ,\ i=2,\ldots,r,
\end{equation}
so that 
\begin{equation*}
\frac{1}{\eta} \le \frac{\Lambda_{i-1}}{\Lambda_i}  \le \eta,\quad i=1,\ldots ,r  .
\end{equation*}

We can ignore the term $K^F_{i-2}(\nu_2,\ldots,\nu_{i-1})/q =O_\eta (1/Q)$ and rewrite 
conditions~\eqref{eq19} as
\begin{equation}\label{eq21}
\left\{ \begin{aligned}
& \frac{Q}{2} < q \le Q ,\quad \alpha\beta \equiv 1 \pmod{q}, \\
& Q-q\le \alpha \le q,
\quad \sum\limits_{i=1}^r \frac{1}{L_{i-1}(\alpha/Q,q/Q)L_i (\alpha/Q,q/Q)} \le \eta, \\
& 0\le \beta \le  \min\bigg\{ q,q\bigg( \frac{Q+q}{\alpha} -2 -\bigg\lfloor \frac{q}{\alpha} \bigg\rfloor\bigg) \bigg\},   \\
& \beta \ge q \bigg( \frac{1+L_{i}   (\alpha/Q,q/Q)}{L_{i-1}(\alpha/Q,q/Q)} -2
- \bigg\lfloor \frac{L_i (\alpha/Q, q/Q)}{L_{i-1}(\alpha/Q,q/Q)} \bigg\rfloor \bigg) ,\quad i=2,\ldots,r, \\
& \beta \le  q \bigg( \frac{1-L_{r-1}(\alpha/Q,q/Q)}{L_r (\alpha/Q,q/Q)} -1+\bigg\lfloor \frac{L_{r-1} (\alpha/Q,q/Q)}{L_r (\alpha/Q,q/Q)}\bigg\rfloor \bigg).
\end{aligned}
\right.
\end{equation}

Consider the sets $\cT_k:=\{ (x,y)\in \cT: \kappa (x,y)=k\}$, $k\in \NN$. Due to~\eqref{eq20}, we are only concerned with (nonempty)
$T$-cylinders 
% $\cT_{(k_1,\ldots,k_r)} :=\cT_{k_1} \cap T^{-1} \cT_{k_2} \cap \cdots \cap T^{-(r-1)} \cT_{k_r}$ 
\begin{align*}
 \cT_{(k_1,\ldots,k_r)} :=\cT_{k_1} \cap T^{-1} \cT_{k_2} \cap \cdots \cap T^{-(r-1)} \cT_{k_r}   ,\quad k_i \le 2\eta.
\end{align*}
On such sets, all functions $L_0,L_1,\ldots L_r$ are linear in each variable.
On each (nonempty) set of the form 
% \mbox{$S=\cT_{(k_1,\ldots,k_r)} \cap \{ (x,y): \lfloor L_{i} (x,y)/L_{i-1}(x,y)\rfloor =a_i,i=2,\ldots ,r\} 
% \cap \{ (x,y) :\lfloor L_{r-1}(x,y) /L_r (x,y)\rfloor =a\}$}
\begin{align*}
   S=\cT_{(k_1,\ldots,k_r)} 
   \cap \Big\{ (x,y): \left\lfloor \sdfrac{L_{i} (x,y)}{L_{i-1}(x,y)}\right\rfloor =A_i,\ i=2,\ldots ,r,\
 \left\lfloor \sdfrac{L_{r-1}(x,y)} {L_r (x,y)}\right\rfloor =A\Big\},
\end{align*}
with $1\le k_i \le 2\eta$, $A_i\le \eta$, $A\le \eta$,
the last three inequalities in~\eqref{eq21} can be expressed as  
\begin{align*}
% \begin{aligned}
 \beta &\le q\, \frac{-(2+k_1)\alpha/Q+q/Q+1}{\alpha/Q},\\
\beta &\ge q \,\frac{A(S,i)\alpha/Q+B(S,i)q/Q+C(S,i)}{L_{i-1}(\alpha/Q,q/Q)},\quad i=2.\ldots,r,\\
 \intertext{and respectively} 
 \beta &\le\frac{A(S,r+1)\alpha/Q+B(S,r+1)q/Q+C(S,r+1)}{L_{r}(\alpha/Q,q/Q)},
% \end{aligned}
\end{align*}
for appropriate constants $A(S,i)$, $B(S,i)$, $C(S,i)$. So,  we can apply \cite[Lemma 3.1]{Bo2007} with $T=Q^{1/4}$ for two piecewise $C^1$ functions 
defined on the connected components of the set 
\begin{equation*}
\bigg\{ \alpha \in [Q-q,q] : \sum\limits_{i=1}^r \frac{1}{L_{i-1}(\alpha/Q,q/Q) L_i(\alpha/Q,q/Q)} \le \eta \bigg\} ,
\end{equation*}  
with $q$ fixed, and having both total variation and sup norms $\ll q$. Applying again M\" obius summation and
proceeding as in the cases where  $r=1$ and $r=2$, we complete the proof of the following result.

\begin{theorem}\label{theorem7}
Consider the functions $\Phi,\Psi,\mathfrak{r} :\cT \rightarrow \RR$,
\begin{equation*}
\Phi (x,y)=\frac{1+y}{x}-2-\bigg\lfloor \frac{y}{x}\bigg\rfloor,\qquad
\Psi (x,y)= \frac{1-x}{y} -1 +\bigg\lfloor \frac{x}{y}\bigg\rfloor ,\qquad
\mathfrak{r} (x,y):=\frac{1}{xy} .
\end{equation*}
We have
\begin{equation*} 
\# \cH_{Q,r} (\eta)  =\frac{Q^2}{\zeta (2)} \int_{\frac{1}{2}}^1 \frac{dw}{w} \Area \Omega_{3,r} (w,\eta) +
 O_{\varepsilon,\eta}  (Q^{7/4+\varepsilon}),
\end{equation*}
where $r \ge  3$ and
\begin{equation*}
\Omega_{3,r} (w,\eta) =\left\{  (u,v):
\begin{aligned}
& 1-w \le u \le w,\quad    \sum\limits_{i=1}^r \mathfrak{r} \big( T^{i-1} (u,w)\big) \le  \eta \\
& \max\limits_{1\le i\le r-1}  \big\{ 0, w\Phi \big( T^i (u,w) \big)\} \le  v 
\le  \min\big\{ w,w \Phi (u,w),w\Psi \big( T^{r-1} (u,w)\big) \big\}
\end{aligned}
\right\} .
\end{equation*}

%{ (u,v) : \begin{tabular}{l}
%$\max\{ 1-w,\frac{1}{\eta w}\} \le u \le w,\  0\le  v\le  w,\  \Phi (u,w) +\frac{uv}{w} \le 1,   $ \\
%$\Phi \circ T^{i} (x,y) +K^F_{i-1} \big(\kappa (u,w),\kappa \circ T (u,w),\ldots , \kappa \circ T^{i-2}(u,w)\big)v$ \\
%$\qquad  -  K^F_{i-2} \big( \kappa \circ T(u,w),\ldots ,\kappa \circ T^{i-2}(u,w)\big) \frac{uv}{w} >1,\quad i=1,\ldots,r-1,$ \\
%$\Psi \circ T^r (u,w) +K^F_{r-1} \big( \kappa (u,w) ,\kappa \circ T (u,w),\ldots,\kappa \circ T^{r-2} (u,w)) v $ \\
%$\qquad- K^F_{r-2} \big( \kappa \circ T(u,w) ,\ldots, \kappa \circ T^{r-2} (u,w)\big) \frac{uv}{w} \le 1,$ \\
%$\sum\limits_{i=1}^r \mathfrak{r} \big( T^{i-1} (x,y)\big) \le  \eta $
%\end{tabular}
%\right\} .
%\end{equation*}
\end{theorem}

Theorem~\ref{thm4} follows assembling~\eqref{eq10}, \eqref{eq16}, \eqref{eq17} and Theorem~\ref{theorem7}, together with~\eqref{eq7} and~\eqref{eq8}.

%%%%%%%%%%%%%%%%%%%%%%%%%%%%%%%%%%%%%%%%%%
% \newpage
\appendix
\setcounter{theorem}{0}
\renewcommand{\thetheorem}{\Alph{section}\arabic{theorem}}
\counterwithin{figure}{section}
\numberwithin{equation}{section}

\section{}

Here, we prove that $\SF_Q^*=\{ 0\} \cup \SF_Q$ gives a unimodular partition of $[0,1]$.
We also describe an algorithmic way of constructing these sets. As for the ordinary Farey sequence~$\cF_Q$, the method is based on certain mediant insertions. 
There, insertion occurs when for two consecutive
elements $a_1 /q_1 <a_2/q_2$ in $\cF_Q$ it happens that $q_1+q_2=Q$ (see, e.g., \cite{HW2008}), while here the insertion process
is delayed until for two consecutive elements $a_1/q_1 < a_2/q_2$ in $\SF_{Q_0}$, 
there exists a mediant fraction~$a/q$ between $a_1/q_1$ 
and $a_2/q_2$ that satisfies the condition $h(a/q) > Q_0$.
The tree generated by the fractions thus inserted for $Q\le 10$ is presented in Figure~\ref{FigureTreeQe10}.
\begin{figure}[htb]
\centering
\hfill
\includegraphics[width=0.89\textwidth]{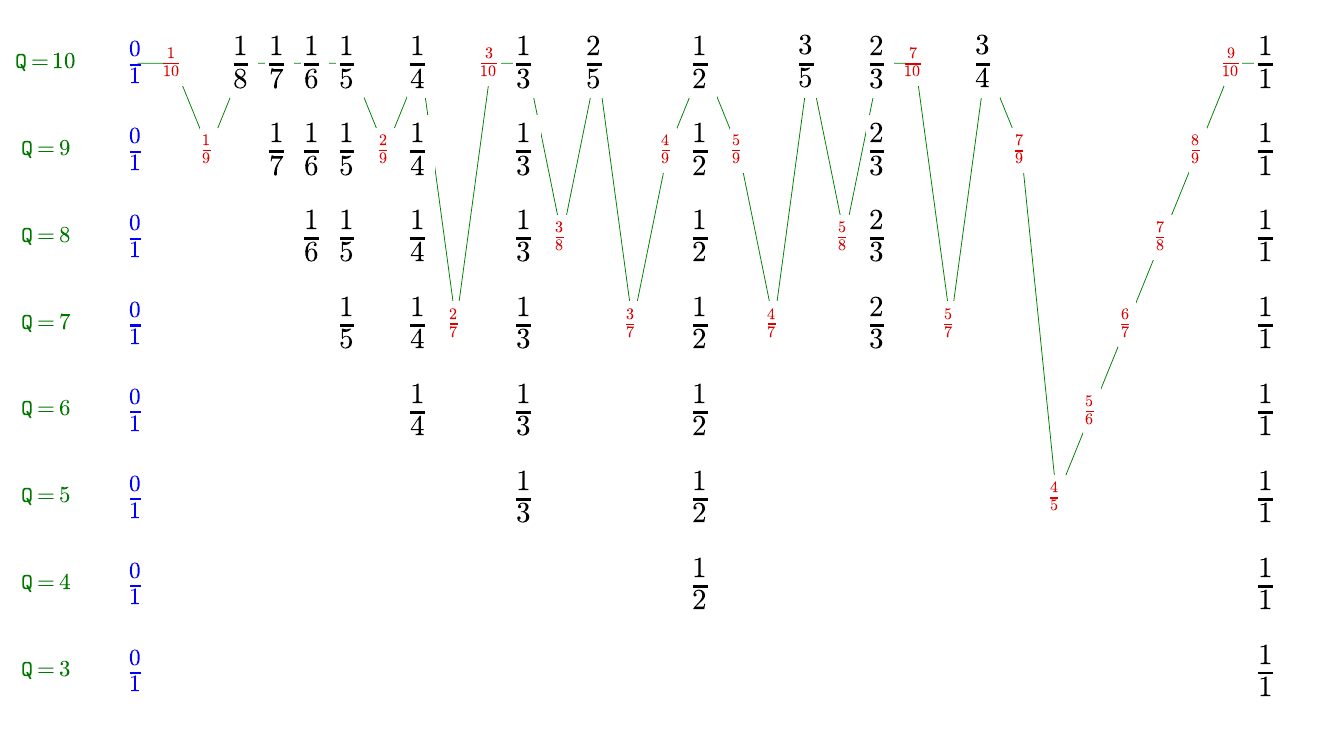}
\hfill\mbox{}
\vspace*{-3mm}
\caption{
The elements of $\SF_Q$, $Q\le 10$.
The Farey fractions are also included in a smaller red font;
they are kept fixed in their insertion position until they are pulled up, 
appearing subsequently at all levels starting from the one where they belong to some~$\SF_Q$.
The support fraction $0/1$ appears at all levels just like the saturated fraction $1/1$  at the opposite end.
}
\label{FigureTreeQe10}
\end{figure}

At the beginning of the proof of Theorem~\ref{theorem1}, we showed that if $0<a<q$ and $(a,q)=1$, then
\begin{equation}\label{eqA1}
\min \{ k: a/q\in \SF_k \} =q+a+\va =: h(a/q).
\end{equation}
Furthermore,
\begin{equation*}
\left( \begin{matrix}
q+\va & q \\ C & a \end{matrix} \right) \in \cS, \quad \text{where } C =\frac{(q+\va)a-1}{q} .
\end{equation*}

\begin{lemma}\label{lemmaA1}
Let $\gamma_1 =a_1 /q_1 < \gamma_2=a_2/q_2$, $0\leq a_i < q_i$, with $a_2 q_1-a_1 q_2 =1$.
Set $Q_i:=h(a_i/q_i)$ and $Q_0 :=\max \{ Q_1,Q_2\}$. Consider the mediant 
$\gamma =\gamma_1 \boxplus \gamma_2 =a/q$, where $a:=a_1+a_2$, $q:=q_1+q_2$, and set 
$Q^\prime :=h(a/q)$. Then
\begin{enumerate}[itemsep=4pt] %leftmargin=20px,
\item[\namedlabel{LemaA1i}{\normalfont{(\textit{i})}}]%\normalfont{(1)}
$Q_0 < Q^\prime$.

\item[\namedlabel{LemaA1ii}{\normalfont{(\textit{ii})}}]%\normalfont{(1)}
If we also assume that $\gamma_1 <\gamma_2$ are consecutive elements in $\SF_{Q_0}^*$, then 
$\gamma_1 < \gamma_2$ are also consecutive in $\SF_Q^*$ if $Q_0 \le Q < Q^\prime$, while
$\gamma_1 <\gamma < \gamma_2$ are consecutive in $\SF_{Q^\prime}^*$.
\end{enumerate}
\end{lemma}

\begin{proof}
For each integer $b$, denote by $\bar{b}$, $\bar{\bar{b}}$, respectively $\bar{\bar{\bar{b}}}$ the
multiplicative inverse of $b\pmod{q_1}$ in $[1,q_1)$, $\pmod{q_2}$ in $[1,q_2)$,
respectively $\pmod{q}$ in $[1,q)$.

Firstly, we show that $Q_1 <Q^\prime$. Recall that $aq_1-a_1 q=1$ and
$1\le \vvva ,q_1 <q$, so $\vvva=q_1$ and 
\begin{equation*}
\begin{aligned}
Q^\prime & =q_1 +q_2 +a_1+a_2 +\vvva = q_1 +q_2 +a_1 +a_2 +q_1 \\
& > q_1 +0+a_1+0 +\va_1 =Q_1.
\end{aligned}
\end{equation*}

We now show that $Q_2 < Q^\prime$. Two cases can occur:

\texttt{Case 1.} $q_1 <q_2$. In this instance, we have $\vva_2 =q_1 =\vvva$. Therefore,
\begin{equation*}
\begin{aligned}
Q^\prime & =q_1+q_2+a_1 +a_2 +\vvva = q_1+q_2 +a_1  + a_2 +\vva_2 \\
& > 0+q_2+0+a_2 +\vva_2 =Q_2.
\end{aligned}
\end{equation*}

\texttt{Case 2.} $q_2 \le q_1$. In this instance,
\begin{equation*}
\begin{aligned}
Q^\prime & = q_1+q_2+a_1+a_2 +\vvva =
q_1+q_2+a_1+a_2 + q_1 \\
& > 0 + q_2 +0 +a_2 +q_2 = q_2+a_2 +\vva_2 =Q_2.
\end{aligned}
\end{equation*}
Therefore, $Q^\prime > \max\{Q_1,Q_2\}$.
\end{proof}

%%%%%%%%%%%%%%%%%%%%%%%%%%%%%%%%%%%%%%%%%%%%%%%%%%
\begin{proposition}\label{propA2}
If $\gamma_1 < \gamma_2$ are consecutive elements in the set $\SF_Q^*$ for some $Q\ge 3$, then $a_2 q_1-a_1q _2=1$.
\end{proposition}

\begin{proof}
According to equality~\eqref{eqA1},  we have $Q\ge \max\{ h(\gamma_1),h(\gamma_2)\}$.

We start with $\gamma_1=0/1$, $\gamma_2=1/1$, consecutive elements in $\SF_3^* =\{ 0,1\}$.
Their mediant~\mbox{$\gamma=1/2$} has $Q^\prime =h(\gamma)=4 > \max \{ h(\gamma_1),h(\gamma_2)\}=3$
and $0/1 < 1/2 < 1/1$ are consecutive elements in $\SF_4^*$. At the second step, we get new mediants
$1/3$ and $2/3$ with $h(1/3)=5$, $h(2/3)=7$.
Then $0/1 <1/3 <1/2$ are consecutive in $\SF_5^*$ and in $\SF_6^*$, while $1/2<2/3<1/1$ are consecutive
in $\SF_7^*$.

Consider now $\gamma_1 <\gamma_2$ consecutive elements in $\SF_{Q_0}^*$ such that $a_2q_1-a_1q_2=1$.
According to Lemma~\ref{lemmaA1}, $Q_0 <Q^\prime :=h(\gamma_1 \boxplus \gamma_2)$,
$\gamma_1 <\gamma_2$ are consecutive elements in $\SF_Q^*$ if $Q_0 \le Q <Q^\prime$, and
$\gamma_1 < \gamma_1 \boxplus \gamma_2 < \gamma_2$ are consecutive elements 
in $\SF_{Q^\prime}^*$. Note that
\begin{align*}
    (a_1+a_2)q_1-a_1(q_1+q_2)=1=a_2 (q_1+q_2)-(a_1+a_2)q_2.
\end{align*}
% $(a_1+a_2)q_1-a_1(q_1+q_2)=1=a_2 (q_1+q_2)-(a_1+a_2)q_2$. 

Next, consider $Q^{\prime\prime}:=h\left(\gamma_1 \boxplus(\gamma_1\boxplus\gamma_2)\right)$. Then, for $Q^{\prime} \leq Q < Q^{\prime\prime}$: $\gamma_1 < \gamma_1 \boxplus \gamma_2$ are consecutive elements in $\SF_{Q}^*$; while for $Q=Q^{\prime\prime}$: $\gamma_1 < \gamma_1 \boxplus (\gamma_1 \boxplus \gamma_2) < \gamma_1 \boxplus \gamma_2$ are consecutive elements in $\SF_{Q}^{*}$.

Similarly, consider $Q^{\prime\prime\prime}:=h(\left(\gamma_1\boxplus\gamma_2\right)\boxplus\gamma_2)$.~Then, for $Q^{\prime} \leq Q < Q^{\prime\prime\prime}$: $\gamma_1 \boxplus \gamma_2 < \gamma_2$ are consecutive elements in~$\SF_{Q}^*$; while for $Q=Q^{\prime\prime\prime}$: $\gamma_1 \boxplus \gamma_2 < (\gamma_1 \boxplus \gamma_2) \boxplus \gamma_2 < \gamma_2$ are consecutive elements in~$\SF_{Q}^{*}$.

Starting with $0/1$ and $1/1$ and applying recursively this procedure, we generate all elements
of~$\SF_Q$ for some fixed $Q\ge 3$ and also prove that $\SF_Q^*$ is a 
unimodular partition of $[0,1]$.
\end{proof}

%\bibliographystyle{plainurl}% shows urls
%\bibliography{FractionsBib}% common bib file

%%%%%%%%%%%%%%%%%%%%%%%%%%%%%%%%%%%%%%%%%%%%%%%%%%%%%%%%%
%\newpage
%\mbox{}
%\end{document}

%%%%%%%%%%%%%%%%%%%%%%%%%%%%%%%%%%%%%%%%%%%%%%%%
%\bibliographystyle{plainurl}% shows urls
%\bibliography{FractionsBib}% common bib file

\begin{thebibliography}{NNNN}

\bibitem{Apo1976}
T. M. Apostol, \emph{Introduction to analytic number theory},
Undergraduate Texts in Mathematics,
 Springer-Verlag, New York-Heidelberg, 1976.
  

\bibitem{AC2014}
J. Athreya and  Y. Cheung, \emph{A Poincar\' e section for the horocycle flow on the space of lattices},
IMRN \textbf{2014} No. 10 (2014), 2643--2690.
% Zbl 1351.37145
\url{https://doi.org/10.1093/imrn/rnt003}

\bibitem{ABCZ2001} 
V. Augustin, F. P. Boca, C. Cobeli, and A. Zaharescu, \emph{The $h$-spacing distribution between Farey points}, Math. Proc. Cambridge Philos. Soc. \textbf{131} (2001), 23--38.
% Zbl 1161.11312
\url{https://doi.org/10.1017/S0305004101005187}

\bibitem{BH2011}
D. A. Badziahin and A. K. Haynes, \emph{A note on Farey fractions with denominators in arithmetic progressions},
Acta Arith. \textbf{147} (2011), 205--215.
\url{https://doi.org/10.4064/aa147-3-1}

\bibitem{Bo2007}
F. P. Boca, \emph{Products of matrices {$\left[\smallmatrix 1 & 1\\ 0 &
              1\endsmallmatrix\right]$} and {$\left[\smallmatrix 1 & 0\\ 1 &
              1\endsmallmatrix\right]$} and the distribution of reduced
              quadratic irrationals},
J. Reine Angew. Math. \textbf{606} (2007), 149--165.
 \url{https://doi.org/10.1515/CRELLE.2007.038}

\bibitem{BCZ2000}
F. P. Boca, C. Cobeli, and A. Zaharescu, \emph{Distribution of lattice points visible from the origin},
Comm. Math. Phys. \textbf{213} (2000), 433--470.
% Zbl 0989.11049
\url{https://doi.org/10.1007/s002200000250}

\bibitem{BCZ2001}
F. P. Boca, C. Cobeli, and A. Zaharescu, \emph{A conjecture of R. R. Hall on Farey points}, J. Reine Angew. Mathematik \textbf{535} (2001), 207--236.
% Zbl 1006.11053
\url{https://doi.org/10.1515/crll.2001.049}

\bibitem{BHS2013}
F. P. Boca, B. Heersink, and P. Spiegelhalter,
\emph{Gap distribution of Farey fractions under some divisibility constraints},
Integers \textbf{13} (2013), paper A44, 15 pages.
\url{https://doi.org/10.5281/zenodo.10199725}

\bibitem{BZ2005}
F. P. Boca and A. Zaharescu, \emph{The correlations of Farey fractions}, J. London Math. Soc. \textbf{72} (2005), 25--39.
% Zbl 1089.11037
\url{https://doi.org/10.1112/S0024610705006629}


\bibitem{CZ2015}
C. Cobeli and A. Zaharescu,  \emph{On the geometry behind a recurrent relation},
 Carpathian J. Math. \textbf{31} (2015), 165--172.

\bibitem{Hall1970}
R. R. Hall, \emph{A note on Farey series}, J. London Math. Soc. \textbf{2} (1970), 139--148.
% Zbl 0191.33202
\url{https://doi.org/10.1112/jlms/s2-2.1.139}

\bibitem{HT1984}
R. R. Hall and G. Tenenbaum, \emph{On consecutive Farey arcs}, Acta Arith.
\textbf{44} (1984), 397--405.
 \url{https://doi.org/10.4064/aa-44-4-397-405}

\bibitem{HW2008}
   G. H. Hardy and E. M. Wright,
  \emph{An introduction to the theory of numbers} (Sixth edition),
 Revised by D. R. Heath-Brown and J. H. Silverman,
              With a foreword by Andrew Wiles,
Oxford University Press, Oxford, 2008.


\bibitem{Hee2019}
B. Heersink, \emph{Distribution of the periodic points of the Farey map (with an appendix by F. P. Boca, B. Heersink and C. Merriman)}, Comm. Math. Phys. \textbf{365} (2019), 971--1003.
\url{https://doi.org/10.1007/s00220-019-03283-0}

\bibitem{He2021}
B. Heersink, \emph{Equidistribution of Farey sequences on horospheres of covers of $\operatorname{SL}(n+1,\ZZ)\backslash \operatorname{SL}(n+1,\RR)$ and applications}, Ergodic Theory \& Dynamical Systems \textbf{41} (2021), 471--493.
% Zbl 1455.37003
\url{https://doi.org/10.1017/etds.2019.71}

\bibitem{KOPS2001}
J. Kallies, A. \"Ozl\"uk, M. Peter, and C. Snyder, \emph{On asymptotic properties of a number theoretic function
              arising out of a spin chain model in statistical mechanics},
 Comm. Math. Phys. \textbf{222} (2001), 9--43.
\url{https://doi.org/10.1007/s002200100495}


\bibitem{KZ1997}
P. Kargaev and A. Zhigljavsky, \emph{Asymptotic distribution of the distance function to the Farey points},
J. Number Theory \textbf{65} (1997), 130--149.
\url{https://doi.org/10.1006/jnth.1997.2110}

\bibitem{Lu2022}
C. Lutsko, \emph{Farey sequences for thin groups}, IMRN \textbf{202} (2022), No. 15, 11642--11689.
% Zbl 1498.11063
\url{https://doi.org/10.1093/imrn/rnab036}

\bibitem{Ma2013}
J. Marklof, \emph{Fine-scale statistics for multidimensional Farey sequence}, in: Limit Theorems in Probability, Statistics and Number Theory, Springer Proceedings in Mathematics \& Statistics \textbf{42} (2013), pp. 49--57.
% Zbl 1291.37044
\url{https://doi.org/10.1007/978-3-642-36068-8_3}

\bibitem{Mik1949}
  M. Mikol\'as, \emph{Farey series and their connection with the prime number  problem. {I}},
 Acta Univ. Szeged. Sect. Sci. Math.,
 \textbf{13} (1949), 93-117.

\bibitem{Pol1986}
M. Pollicott, \emph{Distribution of closed geodesics on the modular surface and quadratic irrationals},
Bull. Soc. Math. France \textbf{114} (1986), 431--446.
\url{http://www.numdam.org/item?id=BSMF_1986__114__431_0}

\bibitem{Ust2013}
A. V. Ustinov, \emph{Spin chains and Arnold's problem on Gauss-Kuzmin statistics for quadratic irrationalities},
Mat. Sb. \textbf{204} (2013), 143--160.
\url{https://doi.org/10.1070/sm2013v204n05abeh004319}

\bibitem{Zhang2021}
X. Zhang, \emph{The gap distribution of directions in some Schottky groups}, J. Mod. Dynamics \textbf{11} (2017), 477--499.
% Zbl 1432.37075
\url{https://doi.org/10.3934/jmd.2017019}



\end{thebibliography}

%%%%%%%%%%%%%%%%%%%%%%%%%%%%%%%%%%%%%%%%%%%%%%%%%%%%%%%%%
%\newpage
\mbox{}
\end{document}